\documentclass[12pt]{amsart}
\usepackage{amsmath,amsthm,amsfonts,amssymb,amscd,euscript,enumerate}
\usepackage{color,listings,graphics,tikz,epsfig} 
\usepackage[colorlinks=true,pagebackref,hyperindex]{hyperref}
\input{xy}
\xyoption{all}

\newlength{\defbaselineskip}
\theoremstyle{plain}
\newtheorem{theorem}{Theorem}[section]

\newtheorem{lemma}[theorem]{Lemma}

\theoremstyle{definition}
\newtheorem{definition}[theorem]{Definition}

\newtheorem{example}[theorem]{Example}

\newtheorem{conjecture}[theorem]{Conjecture}

\newcommand{\C}{\mathcal{C}}
\newcommand{\Z}{\mathbb{Z}}
\newcommand{\I}{\mathbb{I}}
\newcommand{\R}{\mathbb{R}}

\newcommand{\arrowdownto}{\!\!\searrow\!\!} 
\newcommand{\arrowupto}{\!\!\nearrow\!\!} 
\newcommand{\ceil}[1]{\lceil #1 \rceil}

\DeclareMathOperator{\area}{area}
\DeclareMathOperator{\arm}{arm}
\DeclareMathOperator{\bounce}{bounce}
\DeclareMathOperator{\Cat}{Cat}
\DeclareMathOperator{\defc}{defc}
\DeclareMathOperator{\dg}{dg}
\DeclareMathOperator{\dinv}{dinv}
\DeclareMathOperator{\leg}{leg}
\DeclareMathOperator{\Par}{Par}
\DeclareMathOperator{\wt}{wt}
\newcommand{\DV}{\mathcal{DV}} 
\newcommand{\dvmap}{\textsc{dv}}
\newcommand{\DP}{\mathcal{DP}} 
\newcommand{\dpmap}{\textsc{dp}}
\newcommand{\dyckmap}{\textsc{dyck}}

\numberwithin{equation}{section}

\definecolor{darkgreen}{rgb}{0.0, 0.7, 0.0}

\definecolor{cyan}{cmyk}{1,0,0,0}

\newlength{\cellsize}
\newcommand\tableau[1]{
\vcenter{
\let\\=\cr
\baselineskip=-16000pt
\lineskiplimit=16000pt
\lineskip=0pt
\halign{&\tableaucell{##}\cr#1\crcr}}}

\newcommand{\tableaucell}[1]{{%
\def \arg{#1}\def \void{}%
\ifx \void \arg
\vbox to \cellsize{\vfil \hrule width \cellsize height 0pt}%
\else
\unitlength=\cellsize
\begin{picture}(1,1)
\put(0,0){\makebox(1,1){$#1$}}
\put(0,0){\line(1,0){1}}
\put(0,1){\line(1,0){1}}
\put(0,0){\line(0,1){1}}
\put(1,0){\line(0,1){1}}
\end{picture}%
\fi}}
\begin{document}

\title{A Combinatorial Approach to the Symmetry of $q,t$-Catalan Numbers}
\subjclass[2010]{05A19, 05A17, 05E05}
\date{\today} 

\author{Kyungyong Lee}
\address{Department of Mathematics \\
 University of Nebraska --- Lincoln \\
 Lincoln, NE 68588, USA \\
and Korea Institute for Advanced Study \\
Seoul 02455, Republic of Korea}
\email{klee24@unl.edu; klee1@kias.re.kr}
\thanks{The first author was supported by the Korea Institute for
Advanced Study (KIAS), the AMS Centennial Fellowship, NSA grant
H98230-14-1-0323, and the University of Nebraska --- Lincoln.}

\author{Li Li}
\address{Department of Mathematics and Statistics \\
 Oakland University \\ 
 Rochester, MI 48309}
\email{li2345@oakland.edu} 
\thanks{The second author was partially supported by the 
Oakland University URC Faculty Research Fellowship Award.}

\author{Nicholas A. Loehr}
\address{Department of Mathematics \\
 Virginia Tech \\
 Blacksburg, VA 24061-0123 \\
and Department of Mathematics \\
United States Naval Academy \\
Annapolis, MD 21402-5002}
\email{nloehr@vt.edu} 
\thanks{This work was partially supported by a grant
from the Simons Foundation (\#244398 to Nicholas Loehr).}

\keywords{$q,t$-Catalan numbers, Dyck paths, dinv statistic, joint symmetry,
integer partitions}
\maketitle
\begin{abstract}
The \emph{$q,t$-Catalan numbers} $C_n(q,t)$ are polynomials in $q$ and $t$
that reduce to the ordinary Catalan numbers when $q=t=1$. These polynomials 
have important connections to representation theory, algebraic geometry,
and symmetric functions. Haglund and Haiman discovered combinatorial formulas 
for $C_n(q,t)$ as weighted sums of Dyck paths (or equivalently, integer
partitions contained in a staircase shape). This paper undertakes a 
combinatorial investigation of the joint symmetry property $C_n(q,t)=C_n(t,q)$.
We conjecture some structural decompositions of Dyck objects into
``mutually opposite'' subcollections that lead to a bijective explanation
of joint symmetry in certain cases. 
A key new idea is the construction of infinite chains of partitions 
that are independent of $n$ but induce the joint symmetry for all $n$
simultaneously.  
Using these methods, we prove combinatorially
that for $0\leq k\leq 9$ and all $n$, the terms in $C_n(q,t)$ of total degree 
$\binom{n}{2}-k$ have the required symmetry property.  
\end{abstract}

\section{Introduction}
\label{sec:intro}

\subsection{Background on $q,t$-Catalan Numbers.}
\label{subsec:background-qtcat}

The \emph{$q,t$-Catalan numbers} were introduced in 1996 by Garsia and 
Haiman as part of an ongoing study of Macdonald's symmetric polynomials
and the representation theory of the diagonal harmonics 
modules~\cite{GH-qtcat}. Garsia and Haiman's original definition has the form
\begin{equation}\label{eq:GH-qtcat-def}
 C_n(q,t)=\sum_{\mu\in\Par(n)} r_{\mu}(q,t), 
\end{equation}
where $\Par(n)$ is the set of integer partitions of $n$,
and $r_{\mu}(q,t)$ is a complicated rational function in $q$ and $t$
built up from the arms, legs, coarms, and colegs of the cells
in the Ferrers diagram of $\mu$. (These terms will be defined
below; for general background on partitions and symmetric functions,
we refer the reader to~\cite{macbook,stanvol2}. See Haglund's 
book~\cite{hag-book} for an extensive discussion of $q,t$-Catalan numbers.)

Let $\mu'$ be the conjugate of the partition $\mu$, whose diagram is 
obtained by transposing the diagram of $\mu$. It follows from
the definition of the rational functions $r_{\mu}(q,t)$ that
$r_{\mu'}(q,t)=r_{\mu}(t,q)$. Replacing the summation index $\mu$
by $\mu'$ in~\eqref{eq:GH-qtcat-def}, we deduce the 
\emph{joint symmetry property}
\[ C_n(q,t)=C_n(t,q)\qquad(n\geq 0). \]

Another much less obvious property of the original $q,t$-Catalan numbers 
is that each $C_n(q,t)$ is a \emph{polynomial} in $q$ and $t$ with
\emph{nonnegative integer} coefficients. Garsia and Haiman were able to
prove that the specialization $C_n(q,1)$ was a polynomial obtained by
summing terms $q^{\area(\pi)}$, where $\pi$ is a 
\emph{Dyck path of order $n$} and $\area(\pi)$ is the number of area
cells below the path. (See below for more detailed definitions.)
This suggested that there might be a second statistic $\wt(\pi)$ on 
Dyck paths, such that $C_n(q,t)=\sum_{\pi} q^{\area(\pi)}t^{\wt(\pi)}$.
One such statistic, called the \emph{bounce statistic}, was conjectured
by Haglund in 2003~\cite{hag-qtcatconj}. Upon hearing from Garsia that
Haglund had found a statistic, Haiman quickly discovered another 
statistic called \emph{dinv} that also appeared to work. The two conjectures
were found to be equivalent via a bijection on Dyck paths that sends
the pair of statistics $(\area,\bounce)$ to the pair of statistics
$(\dinv,\area)$.  In particular, this implies that all three statistics
have the same distribution on Dyck paths of order $n$.
Garsia and Haglund were eventually able to prove
that Haglund's combinatorial formula for $C_n(q,t)$ was indeed 
correct. Their proof (announced in~\cite{GHag-PNAS} and given in detail
in~\cite{GHag-qtcatpf}) consists of extremely intricate algebraic
manipulations of symmetric polynomials, making extensive use of 
plethystic calculus. 

Since the original $q,t$-Catalan numbers are jointly symmetric in $q$
and $t$, the same must be true of the various combinatorial formulas
for $q,t$-Catalan numbers. In particular, there must exist a bijection
on Dyck paths of order $n$ interchanging the area and bounce statistics,
and there must exist another bijection interchanging area and dinv.
However, at this time, no one has been able to construct such bijections
valid for all $n$. The main goal of this paper is to investigate
the joint symmetry of $q,t$-Catalan numbers from a combinatorial viewpoint.
Roughly speaking, we prove combinatorially that the homogeneous components
of $C_n(q,t)$ of sufficiently high degree have the desired joint symmetry 
property, and we formulate several conjectures that provide a general framework
for understanding joint symmetry.

\subsection{Definition of $q,t$-Catalan Numbers based on Dyck Vectors.}
\label{subsec:def-dyck-vectors}

To state our results more precisely, we must first review the details
of the combinatorial definition of $q,t$-Catalan numbers based on the
area and dinv statistics. The definition can be formulated in three
equivalent ways, using combinatorial objects called \emph{Dyck vectors},
\emph{Dyck paths}, or \emph{Dyck partitions}, which are all illustrated
in Figure~\ref{fig:object} below. The simplest formulation involves
Dyck vectors (also called \emph{Dyck sequences} or the \emph{area vectors
of Dyck paths}), so we discuss these first.

\begin{definition}
A \emph{Dyck vector} is a finite sequence $v=(v_1,v_2,\ldots,v_n)\in\Z^n$ 
such that $v_1=0$, every $v_i\geq 0$, and $v_{i+1}\leq v_i+1$ for all $i<n$.
The \emph{area} of a Dyck vector is $\area(v)=v_1+v_2+\cdots+v_n$.
The \emph{diagonal inversion count} of a Dyck vector, denoted $\dinv(v)$,
is the number of $i<j$ such that $v_i-v_j\in\{0,1\}$.
Let $\DV_n$ be the set of all Dyck vectors of length $n$.  
Finally, define the \emph{combinatorial $q,t$-Catalan numbers} by setting
\[ \Cat_n(q,t)=\sum_{v\in\DV_n} q^{\area(v)}t^{\dinv(v)}. \]
\end{definition}

\begin{example}
The vector $v=(0,1,1,0,1,2,2,0)$ is a Dyck vector in $\DV_8$
with $\area(v)=7$ and $\dinv(v)=12$. The entries in the vector $v$
count the number of shaded cells in each row of Figure~\ref{fig:object}
below, reading from the bottom of the figure to the top. These cells
lie underneath a \emph{Dyck path}, which is a lattice path 
from $(0,0)$ to $(n,n)$ consisting of north and east steps that never go 
below the line $y=x$. This correspondence between Dyck vectors and
Dyck paths is a bijection. So we sometimes identify these two types
of objects, blurring the distinction between a Dyck vector 
(a list of integers) and the associated Dyck path 
(a sequence of north and east steps).
\end{example}

\begin{figure}[h]
\begin{center}
\epsfig{file=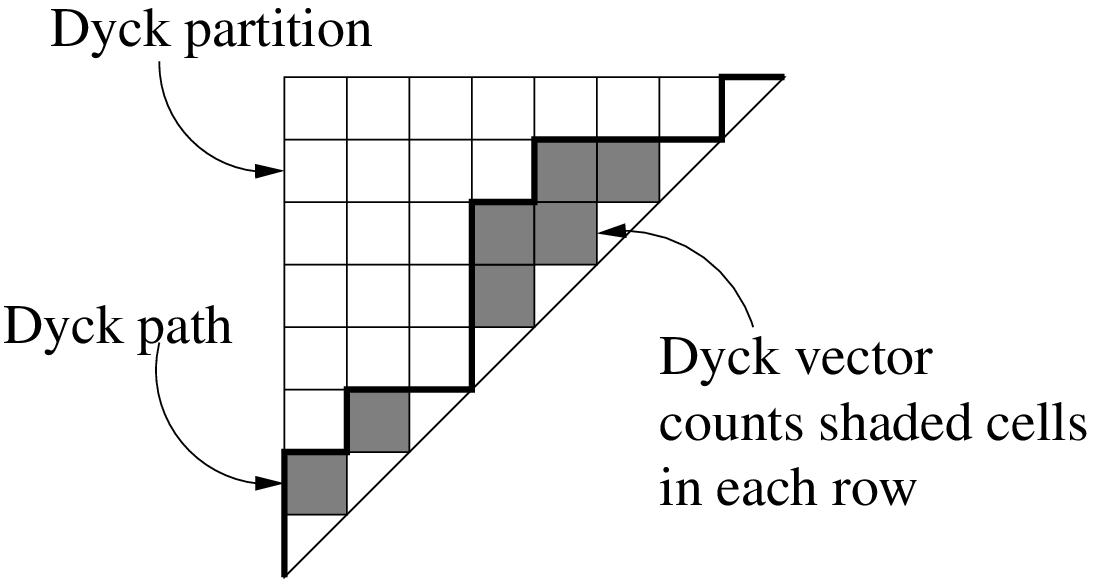,scale=1.0}
\end{center}
\caption{A typical object counted by $\Cat_n(q,t)$.}
\label{fig:object}
\end{figure}

\begin{example}
When $n=3$, we have $\DV_3=\{(0,0,0),(0,0,1),(0,1,0),(0,1,1),(0,1,2)\}$, so
\[ \Cat_3(q,t)=t^3+qt+qt^2+q^2t+q^3. \]
\end{example}

For larger values of $n$, it is convenient to display the coefficients
of $\Cat_n(q,t)$ in a matrix, where the entry $i$ rows to the right
and $j$ columns above the lower-left corner is the coefficient of $q^it^j$
in $\Cat_n(q,t)$. For instance, $\Cat_5(q,t)=t^{10}+q(t^9+t^8+t^7+t^6)+\cdots$
is represented by the matrix
{\footnotesize
\[ \left[\begin{array}{ccccccccccc}
1 & . & . & . & . & . & . & . & . & . & . \\
. & 1 & . & . & . & . & . & . & . & . & . \\
. & 1 & 1 & . & . & . & . & . & . & . & . \\
. & 1 & 1 & 1 & . & . & . & . & . & . & . \\
. & 1 & 2 & 1 & 1 & . & . & . & . & . & . \\
. & . & 1 & 2 & 1 & 1 & . & . & . & . & . \\
. & . & 1 & 2 & 2 & 1 & 1 & . & . & . & . \\
. & . & . & 1 & 2 & 2 & 1 & 1 & . & . & . \\
. & . & . & . & 1 & 1 & 2 & 1 & 1 & . & . \\
. & . & . & . & . & . & 1 & 1 & 1 & 1 & . \\
. & . & . & . & . & . & . & . & . & . & 1 
\end{array}\right], \]
}
where zeroes are denoted by dots for visual clarity.

To prove joint symmetry of $\Cat_n(q,t)$, it is sufficient to show
that each diagonal of the coefficient matrix has the joint symmetry
property. For instance, it is not too hard to show (see Lemma~\ref{lem:C0-opp}
below) that the main (highest)
diagonal consists of $1$'s running southeast from position $(0,\binom{n}{2})$
to position $(\binom{n}{2},0)$. In general, the combinatorial complexity
increases as we move southwest from this main diagonal (in a sense made
more precise later).  The next definition introduces a statistic that keeps 
track of which diagonal each object belongs to.

\begin{definition}
For a Dyck vector $v\in\DV_n$, define the \emph{deficit of $v$} to
be $$\defc(v)=\binom{n}{2}-\area(v)-\dinv(v).$$ (We will see below that
$\defc(v)\geq 0$.) Let $\DV_{n,k}=\{v\in\DV_n:\defc(v)=k\}$. For all $n,k\geq 
0$, define the \emph{$q,t$-Catalan numbers of order $n$ and level $k$} by
\[ \Cat_{n,k}(q,t)=\sum_{v\in\DV_{n,k}} q^{\area(v)}t^{\dinv(v)}. \] 
\end{definition}

\begin{example}
The Dyck vector $v=(0,1,1,0,1,2,2,0)$ has $\defc(v)=\binom{8}{2}-7-12=9$.
Referring to the coefficient matrix above, we find that
\begin{eqnarray*}
\Cat_{5,0}(q,t)&=&q^{10}+qt^9+q^2t^8+\cdots+t^{10}, \\
\Cat_{5,1}(q,t)&=&qt^8+q^2t^7+\cdots+q^8t, \\
\Cat_{5,2}(q,t)&=&qt^7+2q^2t^6+2q^3t^5+2q^4t^4+2q^5t^3+2q^6t^2+q^7t,\\
\Cat_{5,3}(q,t)&=&qt^6+q^2t^5+2q^3t^4+2q^4t^3+q^5t^2+q^6t,\\
\Cat_{5,4}(q,t)&=&q^2t^4+q^3t^3+q^4t^2, 
\end{eqnarray*}
and $\Cat_{5,k}(q,t)=0$ for all $k>4$.  
\end{example}

Our first main result is a combinatorial proof of the following theorem.
\begin{theorem}\label{thm:jsym-levelk}
For all $n\geq 0$ and all $k\leq 9$, $\Cat_{n,k}(q,t)=\Cat_{n,k}(t,q)$.
\end{theorem}

This result will emerge from a framework of (partially proved) conjectures
regarding the combinatorial structure of the sets $\DV_{n,k}$. To state
these conjectures, we must first describe an alternate formulation of the
$q,t$-Catalan numbers phrased in terms of Dyck partitions.

\subsection{Definition of $q,t$-Catalan Numbers based on Dyck Partitions.}
\label{subsec:def-dyck-ptns}

An \emph{integer partition} is a sequence $\mu=(\mu_1,\mu_2,\ldots,\mu_s)$
of weakly decreasing nonnegative integers. We set
$|\mu|=\mu_1+\mu_2+\cdots+\mu_s$ and let $\ell(\mu)$ be the number
of nonzero entries in $\mu$.  We identify sequences that differ only
in the number of trailing zeroes; for example, $(3,3,1)$ and $(3,3,1,0)$
and $(3,3,1,0,0,0)$ all represent the same partition.
Let $\Par(n)$ be the set of partitions $\mu$ with $|\mu|=n$,
and let $\Par=\bigcup_{n=0}^{\infty} \Par(n)$ 
be the set of all integer partitions.

Formally, the \emph{diagram} of a partition $\mu$ is the set 
$$\dg(\mu)=\{(i,j)\in\Z_{>0}^2: 1\leq i\leq\ell(\mu),1\leq j\leq\mu_i\}.$$ 
The \emph{conjugate} of $\mu$ is the partition $\mu'$ with diagram
$\dg(\mu')=\{(j,i):(i,j)\in\dg(\mu)\}$.
We usually visualize $\dg(\mu)$ as an array of left-justified unit boxes,
with $\mu_i$ boxes in the $i$'th row from the top. For example,
the diagram of the partition $\gamma=(7,4,3,3,3,1,0,0)$ appears northwest
of the Dyck path in Figure~\ref{fig:object}. That figure suggests a
correspondence between Dyck paths of order $n$ and integer partitions 
contained in a staircase shape. This is formalized in the next definition.

\begin{definition}
For each $n\geq 0$, let $\Delta_n$ be the partition 
$(n-1,n-2,\ldots,3,2,1,0)\in\Par(\binom{n}{2})$.
We call $\gamma\in\Par$ a \emph{Dyck partition of order $n$} iff
$\dg(\gamma)\subseteq\dg(\Delta_n)$, which means
that $\ell(\gamma)<n$ and $\gamma_i\leq n-i$ for all $1\leq i<n$.
Let $\DP_n$ be the set of Dyck partitions of order $n$.
For $\gamma\in\Par$, let $\Delta(\gamma)$ be the minimum $n$
such that $\gamma\in\DP_n$.
\end{definition}

The set of all integer partitions is the increasing union of the
sets $\DP_n$, i.e.,
\begin{equation}\label{eq:DP-inclusions}
 \DP_0\subseteq\DP_1\subseteq\DP_2\subseteq\cdots\subseteq
   \DP_n\subseteq\cdots\subseteq\Par=\bigcup_{n=0}^{\infty} \DP_n. 
\end{equation}
We obtain a bijection from $\DP_n$ to the set of Dyck paths of order $n$
by mapping $\gamma\in\DP_n$ to the frontier of $\dg(\gamma)$ when
we embed the diagram of $\gamma$ in the diagram of $\Delta_n$.
We can then transform this Dyck path into a Dyck vector by counting
area cells to the right of the path in each row. It is routine to establish
the following formulas for bijections $\dpmap_n:\DV_n\rightarrow\DP_n$
and $\dvmap_n=\dpmap_n^{-1}:\DP_n\rightarrow\DV_n$, which will be needed later:
\begin{eqnarray*}
\dpmap_n(v_1,v_2,\ldots,v_n)&=& 
(n-1-v_n,n-2-v_{n-1},\ldots,n-i-v_{n-i+1},\ldots,1-v_2,0-v_1); \\
\dvmap_n(\gamma_1,\gamma_2,\ldots,\gamma_n) &=&
(0-\gamma_n,1-\gamma_{n-1},\ldots,i-\gamma_{n-i},\ldots,n-1-\gamma_1).
\end{eqnarray*}
In the second formula, we pad $\gamma$ with zero parts so that 
$\gamma$ has exactly $n$ nonnegative parts. 

Next we define area and dinv statistics for Dyck partitions.
First we review the notions of arm, leg, coarm, and coleg for
cells in a partition diagram.

\begin{definition}
Given $\gamma\in\Par$ and a cell $c=(i,j)\in\dg(\gamma)$,
the \emph{arm} of $c$ is $\arm(c)=\gamma_i-j$, which is the
number of boxes strictly right of $c$ in its row of the diagram.
The \emph{leg} of $c$ is $\leg(c)=\gamma'_j-i$, which is the
number of boxes strictly below $c$ in its column of the diagram.
The \emph{coarm} of $c$ is $\arm'(c)=j-1$, which is the
number of boxes strictly left of $c$ in its row of the diagram.
The \emph{coleg} of $c$ is $\leg'(c)=i-1$, which is the
number of boxes strictly above $c$ in its column of the diagram.  
\end{definition}

For example, cell $c=(2,1)$ in the Dyck partition $\gamma$ shown in
Figure~\ref{fig:object} has $\arm(c)=3$, $\leg(c)=4$, $\arm'(c)=0$,
and $\leg'(c)=1$.

\begin{definition}\label{DP*}
Fix a Dyck partition $\gamma\in\DP_n$. Let $\dinv(\gamma)$
be the number of cells $c\in\dg(\gamma)$ with $\arm(c)-\leg(c)\in\{0,1\}$.
Let $\area_n(\gamma)=\binom{n}{2}-|\gamma|$, which is the
number of cells in the skew partition $\Delta_n/\gamma$.
Define the deficit $\defc(\gamma)=\binom{n}{2}-\area_n(\gamma)-\dinv(\gamma)
=|\gamma|-\dinv(\gamma)\geq 0$, which can also be described as
the number of cells $c\in\dg(\gamma)$ with $\arm(c)-\leg(c)\not\in\{0,1\}$.
Let $\DP_{n,k}=\{\gamma\in\DP_n:\defc(\gamma)=k\}$ be
the set of Dyck partitions of \emph{order $n$ and level $k$}.
We write $\DP_{\ast,k}=\bigcup_{n=0}^{\infty} \DP_{n,k}$ for the set of 
all Dyck partitions of level $k$ (of any order).
\end{definition}

For each $k\geq 0$, we have inclusions 
\begin{equation}\label{eq:DPk-inclusions}
 \DP_{0,k}\subseteq\DP_{1,k}\subseteq\DP_{2,k}\subseteq\cdots\subseteq
   \DP_{n,k}\subseteq\cdots\subseteq\DP_{\ast,k}
   =\bigcup_{n=0}^{\infty} \DP_{n,k}. 
\end{equation}

\begin{example}
The partition $\gamma=(5,5,4,3,1,1,0)$ has $|\gamma|=19$,
$\dinv(\gamma)=14$, $\defc(\gamma)=5$, $\Delta(\gamma)=7$,
$\area_7(\gamma)=2$, $\area_8(\gamma)=9$, $\area_9(\gamma)=17$, and so on.
The diagram $\dg(\gamma)$ is shown below; the 14 cells $c$ such 
that $\arm(c)-\leg(c)\in\{0,1\}$ are marked with asterisks.
\[ \tableau {
{}&{\ast}&{}&{}&{}\\
{\ast}&{\ast}&{\ast}&{\ast}&{\ast}\\
{\ast}&{\ast}&{\ast}&{\ast}&\\
{\ast}&{\ast}&{\ast}&\\
{}\\
{\ast} } \]
We have $\dvmap_7(\gamma)=v=(0,0,1,0,0,0,1)$; note that
$\area(v)=2=\area_7(\gamma)$ and $\dinv(v)=14=\dinv(\gamma)$.
\end{example}

One may check that for all $n,k\geq 0$,
\[ \Cat_{n,k}(q,t)=
\sum_{\gamma\in\DP_{n,k}} q^{\area_n(\gamma)}t^{\dinv(\gamma)}\mbox{ and }
\Cat_n(q,t)= \sum_{\gamma\in\DP_{n}} q^{\area_n(\gamma)}t^{\dinv(\gamma)}. \]
It suffices to verify that the bijection $\dvmap_n:\DV_n\rightarrow\DP_n$ 
sends the statistics
$(\area,\dinv,\defc)$ on Dyck vectors to the corresponding statistics
$(\area_n,\dinv,\defc)$ on Dyck partitions. This certainly holds for
the area statistics, and it is not too hard to prove for the dinv
statistics (for details, see~\cite[Lemma 4.4.1 and 6.3.3]{HHLRU}). 
It then follows that the deficit statistic is also preserved, 
which also establishes the earlier claim that
the deficit of a Dyck vector is always nonnegative.

One advantage of using Dyck partitions is that the dinv and deficit
of a partition $\gamma\in\DP_n$ do not change when $n$ is increased,
and the area changes in a predictable way: $\area_{n+1}(\gamma)
=\area_n(\gamma)+n$. The chain of set inclusions~\eqref{eq:DP-inclusions} 
for Dyck partitions translates into a chain of inclusion maps
\begin{equation}\label{eq:DV-inclusions}
\DV_0\stackrel{\iota_0}{\hookrightarrow}
\DV_1\stackrel{\iota_1}{\hookrightarrow}
\DV_2\stackrel{\iota_2}{\hookrightarrow}\cdots
\stackrel{\iota_{n-1}}{\hookrightarrow}
\DV_n\stackrel{\iota_n}{\hookrightarrow}
\DV_{n+1}\stackrel{\iota_{n+1}}{\hookrightarrow}\cdots,
\end{equation}
for Dyck vectors, where $\iota_n(v_1,\ldots,v_n)=(0,v_1+1,v_2+1,\ldots,v_n+1)$
for $(v_1,\ldots,v_n)\in\DV_n$. For $v\in\DV_n$,
one immediately verifies that $\dpmap_{n+1}(\iota_n(v))=\dpmap_n(v)$,
which says that the inclusions in~\eqref{eq:DP-inclusions}
and~\eqref{eq:DV-inclusions} correspond under the bijection
between Dyck vectors and Dyck partitions.
Also note that if $\gamma\in\DP_n$ has associated Dyck vector 
$(v_1,\ldots,v_n)$, then $\Delta(\gamma)=n$ iff $v_j=0$ for some $j>1$.

\subsection{Opposite Objects and Opposite Subsets}
\label{subsec:opposites}

The following terminology will be helpful for discussing joint symmetry.

\begin{definition}
Two Dyck vectors $v,w\in\DV_n$ are called a 
\emph{pair of opposites} iff $\area(v)=\dinv(w)$ and $\dinv(v)=\area(w)$.
Two subsets $V,W\subseteq\DV_n$ are called \emph{opposite to each other} 
iff there exists a unique bijection $f:V\rightarrow W$ such that
$v$ and $f(v)$ are a pair of opposites for all $v\in V$.
(The uniqueness of the bijection implies that for $v\neq v'$ in $V$,
 $(\area(v),\dinv(v))\neq (\area(v'),\dinv(v'))$.)
Two Dyck partitions $\beta,\gamma\in\DP_n$ are called a
\emph{pair of $n$-opposites} iff $\area_n(\beta)=\dinv(\gamma)$
and $\dinv(\beta)=\area_n(\gamma)$.
Two subsets $B,C\subseteq\DP_n$ are called \emph{$n$-opposite to each other}
iff there exists a unique bijection $g:B\rightarrow C$ such that
$\gamma$ and $g(\gamma)$ are a pair of $n$-opposites for all $\gamma\in B$.  
In this definition, we allow $v=w$, $V=W$, $\beta=\gamma$, and $B=C$.
\end{definition}

\begin{example}\label{ex:special-path}
We define an injective map $\dyckmap:\Par(n)\rightarrow\DV_n$ as follows.
For $\lambda=(\lambda_1\geq\lambda_2\geq\cdots)\in\Par(n)$, let 
$\dyckmap(\lambda)\in\DV_n$ be the sequence consisting of $\lambda_1$ 
zeroes, followed by $\lambda_2$ ones, 
$\lambda_3$ twos, and so on. For instance, when $\lambda=(4,2,1,1,1)\in\Par(9)$,
$\dyckmap(\lambda)=(0,0,0,0,1,1,2,3,4)\in\DV_9$. We claim that for all 
partitions $\lambda$, $\dyckmap(\lambda)$ and $\dyckmap(\lambda')$ 
are a pair of opposites.  To see this, note that
\[ \area(\dyckmap(\lambda))=0\lambda_1+1\lambda_2+2\lambda_3+\cdots
  =\sum_{c\in\dg(\lambda)} \leg'(c); \]
\[ \dinv(\dyckmap(\lambda))=\binom{\lambda_1}{2}+\binom{\lambda_2}{2}
 +\binom{\lambda_3}{2} +\cdots =\sum_{c\in\dg(\lambda)} \arm'(c). \]
The natural bijection $\dg(\lambda)\rightarrow\dg(\lambda')$ given by
$(i,j)\mapsto (j,i)$ interchanges coarms and colegs, so that
$\area(\dyckmap(\lambda'))=\dinv(\dyckmap(\lambda))$ and
$\dinv(\dyckmap(\lambda'))=\area(\dyckmap(\lambda))$, as needed.  
\end{example}

One corollary to Theorem~\ref{thm:jsym-levelk} 
and Example~\ref{ex:special-path} is a combinatorial proof of 
the joint symmetry $C_n(q,t)=C_n(t,q)$ for all $n\leq 7$.
This follows since every Dyck vector $v$ of length at most $7$
satisfies $\defc(v)\leq 9$, with four exceptions. These four
exceptional objects are $\dyckmap((3,3,1))$, $\dyckmap((3,3,1)')$,
$\dyckmap((3,2,1,1))$, and $\dyckmap((3,2,1,1)')$, which form
two pairs of opposite objects by the preceding example.

\subsection{Conjectural Decompositions of Level $k$ Objects.}
\label{subsec:intro-conjs}

We can now describe our conjectured structural decomposition of
the sets $\DV_{n,k}$.

\begin{conjecture}\label{conj:DV}
Fix an integer $k\geq 0$. For each $n\geq 0$ and each $\mu\in\Par(k)$,
there exist (possibly identical) subsets $\DV_{n,\mu}$ and
$\overline{\DV}_{n,\mu}$ of $\DV_{n,k}$ satisfying the following:
\begin{itemize}
\item[(a)] $\DV_{n,k}$ is the disjoint union of the sets
$\{\DV_{n,\mu}:\mu\in\Par(k)\}$, and $\DV_{n,k}$ is also the disjoint union
of the sets $\{\overline{\DV}_{n,\mu}:\mu\in\Par(k)\}$.
\item[(b)] $\DV_{n,\mu}$ and $\overline{\DV}_{n,\mu}$ are opposite to 
 each other.
\item[(c)] If $\DV_{n,\mu}$ contains $\dyckmap(\lambda)$,
 then $\overline{\DV}_{n,\mu}$ contains $\dyckmap(\lambda')$.  
\item[(d)] $\iota_n(\DV_{n,\mu})\subseteq \DV_{n+1,\mu}$
 and $\iota_n(\overline{\DV}_{n,\mu})\subseteq \overline{\DV}_{n+1,\mu}$,
 where $\iota_n$ is the injection from~\eqref{eq:DV-inclusions}.
\item[(e)] $\lim_{n\rightarrow\infty} |\DV_{n,\mu}|
           =\lim_{n\rightarrow\infty} |\overline{\DV}_{n,\mu}|=\infty$.
\end{itemize}
\end{conjecture}

An even stronger, but more technical, conjecture is stated
in Section~\ref{sec:stronger_conj} below.
We will prove these conjectures for all $k\leq 9$;
this clearly implies Theorem~\ref{thm:jsym-levelk} for these values of $k$.
We also explicitly construct certain sets $\DV_{n,\mu}$ and
$\overline{\DV}_{n,\mu}$ satisfying (b) and (c) when
$\lambda$ is a hook shape (i.e., $\lambda_2\leq 1$) or an almost-hook shape 
(i.e., $\lambda_2=2$ and $\lambda_3\leq 1$).  

Here is an equivalent formulation of Conjecture~\ref{conj:DV}
in terms of Dyck partitions. 

\begin{conjecture}\label{conj:DP}
Fix an integer $k\geq 0$. For each $\mu\in\Par(k)$,
there exist (possibly identical) infinite subsets $\DP_{\mu}$
and $\overline{\DP}_{\mu}$ of $\DP_{\ast,k}$ satisfying the following:
\begin{itemize}
\item[(a)] $\DP_{\ast,k}$ is the disjoint union of the sets
$\{\DP_{\mu}:\mu\in\Par(k)\}$, and $\DP_{\ast,k}$ is also the disjoint union
of the sets $\{\overline{\DP}_{\mu}:\mu\in\Par(k)\}$.
\item[(b)] For all $n\geq 0$, $\DP_{\mu}\cap\DP_n$
and $\overline{\DP}_{\mu}\cap\DP_n$ are $n$-opposite to each other.  
\item[(c)] If $\DP_{\mu}$ contains $\dpmap_n(\dyckmap(\lambda))$
 where $\lambda\in\Par(n)$, then $\overline{\DP}_{\mu}$ contains 
 $\dpmap_n(\dyckmap(\lambda'))$.  
\end{itemize}
\end{conjecture}

To deduce the previous conjecture from this one, 
take $\DV_{n,\mu}=\dvmap_n(\DP_{\mu}\cap\DP_n)$ and
$\overline{\DV}_{n,\mu}=\dvmap_n(\overline{\DP}_{\mu}\cap\DP_n)$.

\subsection{Outline of Paper.}
\label{subsec:outline}

The rest of the paper is structured as follows.
In Section~\ref{sec:hook_sec}, we construct sets 
$\C_{(k)}\subseteq\DP_{\ast,k}$ and prove that
$\C_{(k)}\cap\DP_n$ is $n$-opposite to itself
for all $n,k\geq 0$. We use these sets to show how
to satisfy Conjecture~\ref{conj:DP}(b) and (c) when $\lambda$ is a hook shape.
In Section~\ref{sec:almost_hook_sec}, we construct certain sets
$\C_{(u,v)}\subseteq\DP_{\ast,u+v}$ and prove
that $\C_{(ab-b-1,b-1)}\cap\DP_n$ and    $\C_{(ab-a-1,a-1)}\cap\DP_n$
are $n$-opposites for all $a,b\geq 2$ and $n\geq 0$.  These sets are used to 
prove Conjecture~\ref{conj:DP}(b) and (c) when $\lambda$ is an almost-hook 
shape.  In Section~\ref{sec:stronger_conj}, we state strengthened versions 
of the conjectures above and prove special cases including
the case $k\leq 9$. Some data needed for this proof are given in an Appendix.

\section{Construction for Hook Shapes}
\label{sec:hook_sec}

In this section, we define collections of Dyck partitions
$\C_{(k)}$ such that $\C_{(k)}\cap\DP_n$ is $n$-opposite
to itself for all $n$ and $k$. Moreover, for every partition 
$\lambda\in\Par(n)$ of hook shape, $\dpmap_n(\dyckmap(\lambda))$ 
belongs to some $\C_{(k)}$.

\subsection{The Operator $\nu$.}
\label{subsec:nu}

A fundamental tool for building $n$-opposite sets of Dyck partitions
is the operator $\nu$ defined next.  This operator can be viewed
as a special case
of the map defined by the present authors in~\cite[Definition 8]{LLL}
(the latter map acts on $m$-Dyck vectors satisfying certain conditions;
here $m=1$).

\begin{definition}
Suppose $\gamma$ is a Dyck partition satisfying the condition
$\gamma_1\leq\ell(\gamma)+2$. For such a partition, we define
$\nu(\gamma)=(\ell(\gamma)+1,\gamma_1-1,\gamma_2-1,\ldots,
\gamma_{\ell(\gamma)}-1)$,
which is also a partition by the hypothesis on $\gamma$. 
\end{definition}

\begin{example}
We have $\nu((5,5,4,3,1,1))=(7,4,4,3,2,0,0)=(7,4,4,3,2)$,
$\nu((7,4,4,3,2))=(6,6,3,3,2,1)$, $\nu((6,6,3,3,2,1))=(7,5,5,2,2,1)$,
and so on. On the other hand, $\nu((7,3,1,1))$ is not defined.
The empty partition $\gamma=(0)=()$ has $\gamma_1=0=\ell(\gamma)$,
so $\nu(\gamma)=(1)$.  
\end{example}

Informally, the next lemma shows that applying $\nu$ allows us to
``move one step northwest'' along a diagonal in the coefficient matrix
for $\Cat_n(q,t)$.

\begin{lemma}\label{lem:propertiesnu}
For all $\gamma\in\Par$ such that $\nu(\gamma)$ is defined
and all $n$ such that $\gamma,\nu(\gamma)\in\DP_n$, 
$$\begin{array}{lcl}
\dinv(\nu(\gamma))&=&\dinv(\gamma) + 1;\\
\area_n(\nu(\gamma))&=&\area_n(\gamma) - 1;\\
\defc(\nu(\gamma))&=&\defc(\gamma).  
\end{array}$$ 
\end{lemma}
\begin{proof} 
This is a special case of \cite[Lemma 9]{LLL}, so we only sketch
the proof here.  Using the formulas for $\dvmap_n$ and $\dpmap_n$,
we can describe how $\nu$ acts on Dyck
vectors as follows. Suppose $\gamma\in\DP_n$ has associated Dyck vector
$\dvmap_n(\gamma)=(v_1,v_2,\ldots,v_n)$. Let $a$ be the maximal index
such that $v_a=a-1$; then 
(provided $\nu(\gamma)$ is defined and in $\DP_n$)
\begin{equation}\label{eq:nu-on-dv}
 \dvmap_n(\nu(\gamma))=(v_1,\ldots,v_{a-1},v_{a+1},\ldots,v_n,a-2).
\end{equation}
One can check that $\nu(\gamma)$ is defined iff $a-2\leq v_n+1$, whereas
$\nu(\gamma)\in\DP_n$ iff $\ell(\gamma)<n-1$ iff $a\geq 2$ iff $0\leq a-2$.
Write $v=\dvmap_n(\gamma)$ and $v'=\dvmap_n(\nu(\gamma))$;
we know $\dinv(v)=\dinv(\gamma)$, $\dinv(v')=\dinv(\nu(\gamma))$,
$\area(v)=\area_n(\gamma)$, and $\area(v')=\area_n(\nu(\gamma))$.
We obtain $v'$ from $v$ by deleting $v_a=a-1$ and appending $a-2$ as the 
new last entry of the vector. It is then evident that $\area(v')=\area(v)-1$.
To see how $\dinv(v')$ compares to $\dinv(v)$, note that
the entries preceding $v_a$ are $(0,1,2,\ldots,a-2)$, which
do not cause any diagonal inversions with $v_a=a-1$.
So removing $v_a=a-1$ from $v$ reduces the dinv statistic by the
number of times $a-1$ or $a-2$ occurs in the sequence $(v_{a+1},\ldots,v_n)$,
When we append $a-2$, we increase the dinv statistic by the
number of times $a-1$ or $a-2$ occurs in the sequence
$(v_1,\ldots,v_{a-1},v_{a+1},\ldots,v_n)$. Since
$(v_1,\ldots,v_{a-1})=(0,1,\ldots,a-2)$ by definition of $a$,
and this prefix contains $a-2$ once and does not contain $a-1$,
we see that the net change in dinv is $+1$ when we go from $v$ to $v'$.
The formulas in the lemma follow, recalling that
$\defc(\gamma)=\binom{n}{2}-\area_n(\gamma)-\dinv(\gamma)$.  
\end{proof}

\subsection{The Dyck Partitions $\I_{j,\ell}$.}
\label{subsec:Ijl}

Next we introduce notation leading to the definition of certain 
Dyck partitions, denoted $\I_{j,\ell}$, that play a key role in the sequel.

\begin{definition}
Given integers $a,b$ with $a\geq b$, let $a\arrowdownto b$
denote the decreasing sequence $(a,a-1,a-2,\ldots,b)$. Similarly, for integers $a,b$ with $a\le b$, we denote by $a\arrowupto b$ the sequence $(a,a+1,\dots,b)$. For any integers $c,d\geq 0$,
let $\underline{c}^d$ denote the sequence consisting of $d$ copies of $c$.
Given finite sequences $v$ and $w$ of possibly different lengths,
we write $v+w=(v_1+w_1,v_2+w_2,\ldots)$, where we pad the shorter
sequence with zeroes at the end if needed.
Finally, given integers $k\geq 0$ and $\ell>0$, 
let $r_{k,\ell}=k-\ell\lfloor k/\ell\rfloor$ be the
remainder when $k$ is divided by $\ell$, and define
\[ \I_{k,\ell}=
(\underline{\ell}^{1+\lfloor k/\ell \rfloor}, (\ell-1)\arrowdownto 1)
+(\underline{0}^{1+\lfloor k/\ell \rfloor}, 
  \underline{1}^{r_{k,\ell}},
  \underline{0}^{\ell-1-r_{k,\ell}}). \]
For $k=\ell=0$, let $\mathbb{I}_{k,\ell}=(0)$ be the zero partition.
\end{definition}

\begin{example}\label{ex:Ikl}
When $k=9$ and $\ell=4$, we find that
 $\mathbb{I}_{9,4} =(4,4,4,3,2,1)+(0,0,0,1,0,0)=(4,4,4,4,2,1)$.
For $\lambda\in\Par(n)$, one may check that
\[ \dpmap_n(\dyckmap(\lambda))
  =\sum_{i=1}^{\ell(\lambda)} 
\I_{((\lambda_i-1)\sum_{j=i+1}^{\ell(\lambda)} \lambda_j),\lambda_i-1}. \]
For example, given $\lambda=(4,2,1,1,1)\in\Par(9)$,
\[ \dpmap_9(\dyckmap(\lambda)) = \mathbb{I}_{15,3} + \mathbb{I}_{3,1} 
= (3,3,3,3,3,3,2,1)+(1,1,1,1) = (4,4,4,4,3,3,2,1). \]
For a partition $\lambda=(a,\underline{1}^b)\in\Par(n)$ of hook shape, we find 
that \[ \dpmap_n(\dyckmap(\lambda))=\I_{(a-1)b,a-1}. \] 
\end{example}

We are going to build collections of Dyck partitions by applying $\nu$
repeatedly to the special partitions $\I_{k,\ell}$. The next lemma describes
some properties of the Dyck partitions $\nu^m(\I_{k,\ell})$,
where $\nu^m$ means apply the $\nu$ operator $m$ times.

\begin{lemma}\label{Delta_II}
Fix integers $k,\ell\geq 0$.
\begin{enumerate}
\item For $0\leq m\leq\ell$, $\nu^m(\I_{k,\ell})$ is defined.
\item For $0\leq m\leq\ell$, $\defc(\nu^m(\I_{k,\ell}))=k$,
so $\nu^m(\I_{k,\ell})\in\DP_{\ast,k}$.
\item For $0\leq m\leq\ell$, $\dinv(\nu^m(\I_{k,\ell}))=\binom{\ell+1}{2}+m$.
\item For $0\leq m\leq r_{-k,\ell}$, 
$\Delta(\nu^m(\I_{k,\ell}))=\ell+\lceil k/\ell\rceil+1$.
\item For $1+r_{-k,\ell}\leq m\leq\ell$, 
$\Delta(\nu^m(\I_{k,\ell}))=\ell+\lceil k/\ell\rceil+2$.  
\end{enumerate}
\end{lemma} 
\begin{proof}
Let $n=\ell+\lceil k/\ell\rceil+2$ and
$r=r_{-k,\ell}=-k\bmod\ell=-k+\ell\lceil k/\ell\rceil$. 
One readily checks that
\[ \DV_n(\I_{k,\ell})= (0,1,\underline{2}^{r},
\underline{1}^{\ell-r},2,3,4,\ldots,\lceil k/\ell\rceil+1). \] 
Using~\eqref{eq:nu-on-dv} and the comments following it, we now compute:
\begin{eqnarray*}
\DV_n(\nu(\I_{k,\ell}))&=&
 (0,1,\underline{2}^{r-1},\underline{1}^{\ell-r},2,3,4,
\ldots,\lceil k/\ell\rceil+1,1);\\
\DV_n(\nu^2(\I_{k,\ell}))&=&
 (0,1,\underline{2}^{r-2},\underline{1}^{\ell-r},2,3,4,
\ldots,\lceil k/\ell\rceil+1,\underline{1}^2);\\ 
\ldots&&\ldots\\ 
\DV_n(\nu^r(\I_{k,\ell})) &=&
 (0,1,\underline{1}^{\ell-r},2,3,4,\ldots,
 \lceil k/\ell\rceil+1,\underline{1}^{r});\\ 
\DV_n(\nu^{r+1}(\I_{k,\ell})) &=&
(0,\underline{1}^{\ell-r},2,3,4,\ldots,
\lceil k/\ell\rceil+1,\underline{1}^{r}, 0);\\ 
\ldots&&\ldots\\ 
\DV_n(\nu^{\ell}(\I_{k,\ell})) &=&
 (0,1,2,3,4,\ldots,\lceil k/\ell\rceil+1,\underline{1}^{r}, 
\underline{0}^{\ell-r}).
\end{eqnarray*}
Assertions (1), (4), and (5) are immediate from this computation.
Next, we compute
\begin{eqnarray*}
|\I_{k,\ell}|&=& \binom{\ell+1}{2}+\ell\lfloor k/\ell\rfloor+r_{k,\ell}
=k+\binom{\ell+1}{2}; \\
\dinv(\I_{k,\ell})&=&\dinv(\dvmap_n(\I_{k,\ell}))
  =\binom{\ell-r+1}{2}+\binom{r+1}{2}+r(\ell-r)=\binom{\ell+1}{2}; \\
\defc(\I_{k,\ell})&=&|\I_{k,\ell}|-\dinv(\I_{k,\ell})=k.  
\end{eqnarray*} 
Now (2) and (3) follow from Lemma~\ref{lem:propertiesnu}.  
\end{proof}

\subsection{The Collections $\C_{(k)}$.}
\label{subsec:define-Ck}

\begin{definition}
For each integer $k\geq 0$, define 
\begin{equation}\label{def:hook} 
\C_{(k)} = \bigcup_{\ell=\min(1,k)}^{\infty}
\{\nu^m(\I_{k,\ell}):0\leq m\leq\ell\}\subseteq \DP_{\ast,k}. 
\end{equation}
(The notation $\C_{(k)}$ is chosen so that it is compatible 
with Conjecture~\ref{strong_conj} below.)
\end{definition}

We remark that for all $i\geq\min(1,k)$, there exists a unique 
$\gamma\in\C_{(k)}$ with $\dinv(\gamma)=i$. This follows from
Lemma~\ref{Delta_II}(3) and the readily verified fact
that every integer $i\geq 0$ can be written uniquely in the form
$i=\binom{\ell+1}{2}+m$ for some $\ell,m\geq 0$ with $0\leq m\leq\ell$.

\begin{example}
Table~\ref{tab:calc-C4} shows the first several Dyck partitions in $\C_{(4)}$.  
\begin{table}[ht]
\begin{center}
\begin{tabular}{l|l|l|l}
Dyck partition $\gamma$ & $\dinv(\gamma)$ & $\defc(\gamma)$ & $\Delta(\gamma)$
\\\hline\hline
$\I_{4,1}=(1,1,1,1,1)$ & $1$ & $4$ & $6$ \\
$\nu(\I_{4,1})=(6)$    & $2$ & $4$ & $7$ \\\hline
$\I_{4,2}=(2,2,2,1)$   & $3$ & $4$ & $5$ \\
$\nu(\I_{4,2})=(5,1,1,1)$&$4$& $4$ & $6$ \\
$\nu^2(\I_{4,2})=(5,4)$& $5$ & $4$ & $6$ \\\hline
$\I_{4,3}=(3,3,3,1)$   & $6$ & $4$ & $6$ \\
$\nu(\I_{4,3})=(5,2,2,2)$&$7$& $4$ & $6$ \\
$\nu^2(\I_{4,3})=(5,4,1,1,1)$ & $8$ & $4$ & $6$ \\
$\nu^3(\I_{4,3})=(6,4,3)$& $9$ & $4$ & $7$ \\\hline
$\I_{4,4}=(4,4,3,2,1)$ & $10$ & $4$ & $6$ \\
$\nu(\I_{4,4})=(6,3,3,2,1)$&$11$ & $4$ & $7$  \\
$\ldots$ & $\ldots$ & $\ldots$ & $\ldots$
\end{tabular}
\end{center} 
\caption{Dyck partitions in $\C_{(4)}$.}
\label{tab:calc-C4}
\end{table}
\end{example}

From Example~\ref{ex:Ikl}, we know
that the partition $\lambda=(a+1,\underline{1}^b)\in\Par(n)$ of hook shape
satisfies $\dpmap_{n}(\dyckmap(\lambda))=\I_{ab,a}\in\C_{(ab)}$.
The conjugate partition $\lambda'=(b+1,\underline{1}^a)$ satisfies
$\dpmap_{n}(\dyckmap(\lambda'))=\I_{ba,b}\in\C_{(ab)}$.
Thus, we could satisfy Conjecture~\ref{conj:DP}(c) for this $\lambda$
by setting $\mu=(ab)$ and $\DP_{\mu}=\overline{\DP}_{\mu}=\C_{(ab)}$.
The main result of this section is that this choice of $\DP_{\mu}$
and $\overline{\DP}_{\mu}$ satisfies part (b) of the conjecture:

\begin{theorem}\label{hook_chain}
For all $k,n\geq 0$, $\C_{(k)}\cap\DP_n$ is $n$-opposite to itself.
\end{theorem}

We prove this theorem in the following subsections. As an example
of the theorem, when $k=4$ and $n=6$, we see from Table~\ref{tab:calc-C4} that
\[ \sum_{\gamma\in\C_{(4)}\cap\DP_6} q^{\area_6(\gamma)}t^{\dinv(\gamma)}
  = q^{10}t^1+q^8t^3+q^7t^4+q^6t^5+q^5t^6+q^4t^7+q^3t^8+q^1t^{10}, \]
which is jointly symmetric in $q$ and $t$.
Replacing $n=6$ by $n=5$, the sum becomes $q^3t^3$.

\subsection{Proof of the $k=0$ Case.}
\label{subsec:prove-hook-chain0}

\begin{lemma}\label{lem:C0-opp}
For all $n\geq 0$, $\C_{(0)}\cap\DP_n$ is $n$-opposite to itself.
\end{lemma}
\begin{proof}
Taking $k=0$ in Lemma~\ref{lem:propertiesnu}(4) and (5), we see that
$\Delta(\nu^m(\I_{0,\ell}))$ is $\ell+1$ for $m=0$, and $\ell+2$ for $m>0$.
So, $\nu^m(\I_{0,\ell})\in\DP_n$ iff $\Delta(\nu^m(\I_{0,\ell}))\leq n$
iff either $m=0$ and $\ell\leq n-1$, or $m>0$ and $\ell\leq n-2$. Therefore,
\[ \C_{(0)}\cap\DP_n=\left(\bigcup_{\ell=0}^{n-2}
 \{\nu^m(\I_{0,\ell}):0\leq m\leq\ell\}\right)\cup\{\I_{0,n-1}\}. \]
Now by Lemma~\ref{lem:propertiesnu}(2) and (3) and the definition of $\area_n$,
\[ \{(\dinv(\gamma),\area_n(\gamma)):\gamma\in\C_{(0)}\cap\DP_n\} =
\left\{\left(0,\binom{n}{2}\right),\left(1,\binom{n}{2}-1\right),\ldots,
\left(\binom{n}{2},0\right)\right\}.\]
This list of exponent pairs is jointly symmetric, so
$\C_{(0)}\cap\DP_n$ is $n$-opposite to itself.  
\end{proof}

\subsection{Exponent Pairs Appearing in $\C_{(k)}\cap\DP_n$.}
\label{subsec:exp-pairs-Ck}

In the rest of this section, let $k>0$ be a fixed integer.  
Define two functions $\tau, s:\mathbb{Z}_{>0}\to\mathbb{Z}_{>0}$ 
(depending on $k$) by
\begin{equation}\label{taus}
\tau(\ell)=\lceil k/\ell\rceil,\quad s(\ell)=\ell+\tau(\ell)+1.
\end{equation}
We translate Theorem~\ref{hook_chain} into a question about the sequence 
$A_{(k)}=(a_1,a_2,\ldots)$, defined as follows:
for $i$ of the form $\binom{\ell+1}{2}+m$ where $\ell\geq 1$
and $0\leq m\leq\ell$, let $a_i=\Delta(\nu^m(\I_{k,\ell}))$. 
By Lemma~\ref{Delta_II}(4) and (5), 
we have $$a_i=\begin{cases}
s(\ell),& \text{ for $0\leq m\leq r_{-k,\ell}$;} \\
s(\ell)+1,& \text{ for $1+r_{-k,\ell}\leq m\leq\ell$.} 
\end{cases}$$

\begin{example} 
For $1\leq k\leq 5$, the sequences $A_{(k)}$ begin as follows:
\begin{eqnarray*} 
A_{(1)}&=&(3,4;4,4,5;5,5,5,6;6,6,6,6,7;\ldots); \\
A_{(2)}&=&(4,5;4,5,5;5,5,6,6;6,6,6,7,7;\ldots); \\
A_{(3)}&=&(5,6;5,5,6;5,6,6,6;6,6,7,7,7;\ldots); \\
A_{(4)}&=&(6,7;5,6,6;6,6,6,7; 6,7,7,7,7;\ldots)
\text{\quad[cf. Table~\ref{tab:calc-C4}];} \\
A_{(5)}&=&(7,8;6,6,7;6,6,7,7;7,7,7,7,8;7,8,8,8,8,8;\ldots).
\end{eqnarray*}
When $k=50$, we have
$A_{(50)}=($52, 53; 28, 29, 29; 21, 21, 22, 22; 18, 18, 18, 19, 19; 16, 17, 17, 17, 17, 17; 16, 16, 16, 16, 16, 17, 17; 16, 16, 16, 16, 16, 16, 16, 17; 16, 16, 16, 16, 16, 16, 16, 17, 17; $\ldots)$.
\end{example}

For any $n\in\mathbb{Z}_{>0}$, define a set
$$S_{n,k}=\Big\{(\dinv(\gamma),\area_n(\gamma)): 
\gamma\in \C_{(k)}\cap\DP_n\Big\}.$$ 
Then, by Lemma~\ref{Delta_II}(2) and (3) and the definition of $\area_n$,
$$S_{n,k}=\Big\{\left(i,\binom{n}{2}-k-i\right): i\ge 1, a_i\le n\Big\}.$$ 
Thus Theorem~\ref{hook_chain} is equivalent to the following statement,
proved below: for all $n,k,x,y>0$, $(x,y)\in S_{n,k}$ if and only if 
$(y,x)\in S_{n,k}$.

\subsection{Preliminary Lemmas and Definitions.}
\label{subsec:hook-lem-def}

Define $L=\max\{\ell:\ell\leq\tau(\ell)\}$.
The following lemma states some fundamental facts
about the functions $\tau$ and $s$ defined in~\eqref{taus}.

\begin{lemma}
\begin{enumerate}
\item[(1)] $L=\max\{\ell:\ell^2-\ell<k\}$.
\item[(2)] $\tau$ is a weakly decreasing function, and so
 $s(\ell+1)\leq s(\ell)+1$ for all $\ell$.
\item[(3)] The restriction of $\tau$ to $\{1,2,\ldots,L\}$ is a strictly 
decreasing function, and so $s$ is weakly decreasing on $\{1,2,\ldots,L\}$.
\item[(4)] For all $\ell>L$, $\tau(\ell+1)\geq\tau(\ell)-1$,
and so $s$ is weakly increasing on $\{L+1,L+2,\ldots\}$.
\item[(5)] For all $\ell$, we have $\tau^2(\ell)\leq\ell$,
and so $s(\tau(\ell))\leq s(\ell)$.
\item[(6)] For all $\ell\leq L$, we have $\tau^2(\ell)=\ell$,
and so $s(\tau(\ell))=s(\ell)$.  
\end{enumerate}
\end{lemma}
\begin{proof}
First note that for all $j\in\Z$ and $x\in\R$, $j\leq\ceil{x}$
iff $j<x+1$. So, $\ell\leq\tau(\ell)=\ceil{k/\ell}$ iff $\ell<(k/\ell)+1$
iff $\ell^2-\ell<k$, proving (1). If $\ell<\ell'$, then $k/\ell>k/\ell'$,
hence $\ceil{k/\ell}\geq\ceil{k/\ell'}$, proving (2). 
To prove (3), assume $1<\ell\leq L$, and show $\tau(\ell)<\tau(\ell-1)$. 
By (1), $\ell^2-\ell<k$, which rearranges to $k/\ell<k/(\ell-1)-1$. This 
implies $\ceil{k/\ell}<\ceil{k/(\ell-1)}$, as needed. To prove (4), 
assume $\ell\geq L+1$, so $\ell^2-\ell\geq k$ by (1).  Then $\ell^2+\ell>k$,
which rearranges to $(k/\ell)-1<k/(\ell+1)$. Taking the ceiling of both
sides gives $\tau(\ell)-1\leq\tau(\ell+1)$, as needed. To prove (5),
note that $k/\ell\leq\ceil{k/\ell}$ implies $\ell\geq k/\ceil{k/\ell}$.
Since $\ell$ is an integer, $\ell\geq\ceil{k/\ceil{k/\ell}}=\tau^2(\ell)$, 
as needed. To prove (6), assume $\ell\leq L$; by (5), it suffices to 
prove $\ell\leq\tau^2(\ell)$. Now $\ell^2-\ell\leq k$,
which rearranges to $(k/\ell)+1\leq k/(\ell-1)$. Since
$\ceil{k/\ell}<(k/\ell)+1$, we get $\ceil{k/\ell}<k/(\ell-1)$, which
rearranges to $\ell<k/\ceil{k/\ell}+1$. By the first sentence of this
proof, $\ell\leq\tau^2(\ell)$ follows.  
\end{proof}

The following definitions will help us analyze $S_{n,k}$.
We may assume that $S_{n,k}\neq\emptyset$, which is equivalent
to the condition $\min\{s(\ell):\ell\geq 1\}\leq n$.

\begin{definition}
Let $r_{j,\ell}^+$ be the unique integer in the range $\{1,2,\ldots,\ell\}$
such that $j\equiv r_{j,\ell}^+\pmod{\ell}$. Define 
\begin{flalign*}
&& x_1=\min\{i: a_i\le n\},\quad \quad & x_2=\max\{i: a_i\le n\}
=\max\{i: a_i= n\},&&\\
&& \ell_1=\min\{\ell : s(\ell)\le n\},\quad &\ell_2=\max\{\ell : s(\ell)\le n\}=\max\{\ell : s(\ell)= n\}, && \\
&& \ell_1'=\min\{\ell:s(\ell)<n\}, 
\quad &\ell_2'=\min\{\ell>\ell_1': s(\ell)=n\}\quad\text{(if these exist).} &&
\end{flalign*}
\end{definition}

We need the following remarks about this definition.
The second formulas for $\ell_2$ and $x_2$ follow from
part (2) of the previous lemma and the definition of $a_i$.
The numbers $\ell_1'$ and $\ell_2'$ are undefined iff 
$\min\{s(\ell):\ell\geq 1\}=n$. Using properties of $s$ from
the previous lemma, one may check that 
$x_1=\binom{\ell_1+1}{2}$, $x_2=\binom{\ell_2+1}{2}+r_{-k,\ell_2}$,
$\ell_1\leq L\leq\ell_2$, and (when $\ell_1'$ and $\ell_2'$ are defined)
$\ell_1\leq\ell_1'\leq L<\ell_2'\leq\ell_2$.

\begin{lemma}\label{easylemma} 
\begin{enumerate}
\item[(1)] If $\ell<\tau(\ell)$ and $s(\ell)=n$, then
$r_{-k,\ell}=r_{-k,n-1-\ell}$ and ${r^+_{k,\ell}=r^+_{k,n-2-\ell}}$.
\item[(2)] $\ell_1=\tau(\ell_2)$ and  $x_1+x_2=\binom{n}{2}-k$.
\item[(3)] If $n>\min\{s(\ell):\ell\geq 1\}$, 
then $\ell'_1-\ell_1=\ell_2-\ell_2'$. 
\end{enumerate}
\end{lemma}
\begin{proof}
(1) Since $\lceil k/\ell \rceil=\tau(\ell)=n-1-\ell$,  
we have $k+r=\ell(n-1-\ell)$ for a unique $r$ satisfying $0\le r<\ell$. 
Then $-k=-\ell(n-1-\ell)+r$, implying $-k\equiv r \pmod \ell$ and 
$-k\equiv r \pmod{n-1-\ell}$. Note that $0\le r<\ell<n-1-\ell$, so 
$r_{-k,\ell}=r=r_{-k,n-1-\ell}$.  Next, note that $k=\ell(n-1-\ell)-r
=\ell(n-2-\ell)+(\ell-r)$ and $0<\ell-r\le \ell\le n-2-\ell$, thus 
$r^+_{k,\ell}=\ell-r=r^+_{k,n-2-\ell}$.

(2) To show the first equality: since $s(\tau(\ell_2))\le s(\ell_2)=n$, 
we conclude that $\tau(\ell_2)\ge \ell_1$ by the definition of $\ell_1$.
On the other hand, $\ell_1\le L$ implies $s(\tau(\ell_1))=s(\ell_1)\le n$, 
so we conclude that $\tau(\ell_1)\le \ell_2$ by the definition of $\ell_2$; 
then the weakly decreasing property of $\tau$ implies that  
$\tau^2(\ell_1)\ge \tau(\ell_2)$, i.e., $\ell_1\ge \tau(\ell_2)$.
Thus $\ell_1=\tau(\ell_2)$.

To show the second equality: $s(\ell_2)=n$ implies $\ell_2+\tau(\ell_2)+1=n$, 
thus $\ell_2+\ell_1+1=n$. It follows as in the proof of (1)  
that $k+r_{-k,\ell_2}=\ell_2(n-1-\ell_2)=\ell_1\ell_2$. So
$x_1+x_2+k=\binom{\ell_1+1}{2}+\binom{\ell_2+1}{2}+r_{-k,\ell_2}+k=
\binom{\ell_1+1}{2}+\binom{\ell_2+1}{2}+\ell_1\ell_2
=\binom{\ell_1+\ell_2+1}{2}=\binom{n}{2}$.

(3) Applying (2) to $n-1$ and noticing that 
$\ell_2'-1=\max\{\ell : s(\ell)=n-1\}$, we conclude that 
$\ell_1'=\tau(\ell'_2-1)$. Then $\ell_1'+(\ell'_2-1)+1=s(\ell'_2-1)=
n-1=\ell_1+\ell_2$, thus $\ell_1'-\ell_1=\ell_2-\ell_2'$.
\end{proof}

\subsection{Proof of Symmetry of $S_{n,k}$.}
\label{subsec:symm-Snk}

We are now ready to prove the symmetry of $S_{n,k}$.
First consider the case where $\ell_1'$ and $\ell_2'$ are defined
and $\ell_1<\ell_1'$. Then the sequence $A_{(k)}$ has the following form, 
where $*$ represents a number $>n$, and ${\rm o}$ 
represents a number $\le n$.
\begin{equation}\label{eq:pattern}
\aligned
 &\cdots **;
    \underbrace{n,\ldots,n}_{1+r_{-k,\ell_1}},
    \underbrace{n+1,\ldots,n+1}_{r^+_{k,\ell_1}};
    \cdots\cdots;
    \underbrace{n,\ldots,n}_{1+r_{-k,\ell_1'-1}},
    \underbrace{n+1,\ldots,n+1}_{r^+_{k,\ell_1'-1}};
   {\rm oo}\cdots\\
&\quad\quad\quad    \cdots {\rm oo}; \underbrace{n,\ldots,n}_{1+r_{-k,\ell_2'}},
    \underbrace{n+1,\ldots,n+1}_{r^+_{k,\ell_2'}};
    \cdots\cdots;
      \underbrace{n,\ldots,n}_{1+r_{-k,\ell_2}},
    \underbrace{n+1,\ldots,n+1}_{r^+_{k,\ell_2}};
  **\cdots
\endaligned
\end{equation}
So the subsequence $(a_{i_1},a_{i_1+1},\ldots,a_{i_2})$ of $A_{(k)}$ 
has the form
\begin{equation}\label{general case}
    \underbrace{\rm o\cdots o}_{1+r_{-k,\ell_1}}
    \underbrace{*\cdots *}_{r^+_{k,\ell_1}}
    \cdots\cdots
    \underbrace{\rm o\cdots o}_{1+r_{-k,\ell_1'-1}}
    \underbrace{*\cdots*}_{r^+_{k,\ell_1'-1}}
    \underbrace{\rm
      o\cdots\cdots o
      }_{{\rm central\; block\; } B}
       \underbrace{*\cdots*}_{r^+_{k,\ell_2'}}
        \underbrace{\rm o\cdots o}_{1+r_{-k,\ell_2'+1}}
        \cdots\cdots
       \underbrace{*\cdots*}_{r^+_{k,\ell_2-1}}
      \underbrace{\rm o\cdots o}_{1+r_{-k,\ell_2}}
\end{equation}
We see that $S_{n,k}$ is symmetric if $x_1+x_2=\binom{n}{2}-k$ and the pattern
of {\rm o}'s and $*$'s in \eqref{general case} is equal to its own reversal.
The first condition holds due to Lemma~\ref{easylemma}(2), which also
tells us that $\tau(\ell_2)=\ell_1$. Since $s(\ell_2)=n$, we deduce
that $\ell_2+\ell_1+1=n$. It now follows from Lemma~\ref{easylemma}(1)
that the leftmost and rightmost blocks of o's in~\eqref{general case}
have the same length, the leftmost and rightmost blocks of $*$'s have
the same length, and so on. Part (3) of the lemma ensures that there
are the same number of blocks of o's to the left and right of the central 
block $B$, so that~\eqref{general case} has the required symmetry.

Another case that can occur is when $\ell_1'$ and $\ell_2'$ are defined
and $\ell_1=\ell_1'$ (hence $\ell_2=\ell_2'$ by part (3) of the lemma).
In this case, the $n$'s and $(n+1)$'s explicitly displayed 
in~\eqref{eq:pattern} are absent, and~\eqref{general case} consists of
just the central block of o's. Here it suffices to know that
$x_1+x_2=\binom{n}{2}-k$, which follows from part (2) of the lemma.

A final special case occurs when $\min\{s(\ell):\ell\geq 1\}=n$ 
and hence $\ell_1',\ell_2'$ are undefined. In this case,
\eqref{general case} should be replaced by 
\begin{equation}\label{degenerate case}
    \underbrace{\rm o\cdots o}_{1+r_{-k,\ell_1}}
    \underbrace{*\cdots *}_{r^+_{k,\ell_1}}
    \underbrace{\rm o\cdots o}_{1+r_{-k,\ell_1+1}}
    \cdots\cdots
      \underbrace{\rm o\cdots o}_{1+r_{-k,\ell_2-1}}
       \underbrace{*\cdots*}_{r^+_{k,\ell_2-1}}
      \underbrace{\rm o\cdots o}_{1+r_{-k,\ell_2}},
\end{equation}
which is equal to its reversal by the same argument used above.

\subsection{Bijection on $\C_{(k)}\cap\DP_n$.}
\label{subsec:bij-Ck}

We can use the preceding constructions to describe an explicit bijection
$F_{n,k}$ on $\C_{(k)}\cap\DP_n$ that interchanges the statistics
$\area_n$ and $\dinv$. Suppose $\gamma$ is a Dyck partition in
$\C_{(k)}\cap\DP_n$ with $\area_n(\gamma)=i$ (and hence 
$\dinv(\gamma)=\binom{n}{2}-k-i$). Write $i$ uniquely in the 
form $\binom{\ell+1}{2}+m$ where $\ell,m\geq 0$ and $0\leq m\leq\ell$.
Then define $F_{n,k}(\gamma)=\nu^m(\I_{k,\ell})$, which has dinv equal to
$i$ (and hence $\area_n$ equal to $\binom{n}{2}-k-i$) by parts (2) and (3)
of Lemma~\ref{Delta_II}. The calculations in the preceding subsections
are necessary to know that $F_{n,k}(\gamma)$ is still in the set $\DP_n$.
Evidently, $F_{n,k}$ is an involution and is the unique bijection
on $\C_{(k)}\cap\DP_n$ interchanging the two statistics.

\section{Construction for Almost-Hook Shapes}
\label{sec:almost_hook_sec}

In this section, we construct collections of Dyck partitions
$\C_{(u,v)}$ that contain all the objects $\dpmap_n(\dyckmap(\lambda))$
for almost-hook shapes $\lambda$. The main result is that
$\C_{(ab-b-1,b-1)}\cap\DP_n$ and $\C_{(ab-a-1,a-1)}\cap\DP_n$
are $n$-opposite subsets for all $a,b\geq 2$ and all $n\geq 0$.

\subsection{Removing Cells from Partition Diagrams.}
\label{subsec:remove-cells}

To describe the Dyck partitions in $\C_{(u,v)}$, it is convenient to
introduce the following construction that removes specified cells from
a partition diagram. We totally order $\Z_{>0}^2$ by setting
$(i,j)<(i',j')$ iff either $i+j<i'+j'$, or $i+j=i'+j'$ and $i<i'$.
Given $\gamma\in\Par$ and $k\in\{1,2,\ldots,|\gamma|\}$, 
let $u_k=u_k(\gamma)$ be the $k$'th largest cell $(i,j)$ in the diagram 
of $\gamma$ relative to this total ordering. 

\begin{definition}
Suppose $\gamma\in\Par$ and $u_{i_1},\ldots,u_{i_s}$ are cells in
$\dg(\gamma)$ such that $\dg(\gamma)\setminus\{u_{i_1},\ldots,u_{i_s}\}$
is also the diagram of some partition $\beta$. By a slight abuse of notation,
we define \[ \gamma\setminus\{u_{i_1},\ldots,u_{i_s}\}=\beta. \]
\end{definition}

\begin{example}
Given $\gamma=(4,4,4,4,2,1)$, the cells $u_1,u_2,u_3,u_4$ in $\dg(\gamma)$
are shown below.
\[ \tableau{
{}&{}&{}&{}\\
{}&{}&{}&{}\\
{}&{}&{}&{}\\
{}&{}&{u_4}&{u_1}\\
{}&{u_3}\\
{u_2} } \]
We have $\gamma\setminus\{u_1,u_4\}=(4,4,4,2,2,1)$
and $\gamma\setminus\{u_2,u_3\}=(4,4,4,4,1)$.
\end{example}

\subsection{The Collections $\C_{(u,v)}$.}
\label{subsec:Cuv}

Fix an almost-hook shape $\lambda=(b,2,\underline{1}^{a-2})$,
where $a,b\geq 2$.  From Example~\ref{ex:Ikl}, we have
\begin{eqnarray*}
\dpmap_n(\dyckmap(\lambda))&=&\I_{(b-1)a, b-1} + \I_{a-2,1} \\
&=&(\underline{b}^{a-1},b-1,(b-1)\arrowdownto 1) \\
&=&\I_{b(a-1),b}-(\underline{0}^{a-1},1,\underline{0}^{b-1})
  =\I_{b(a-1),b}\setminus\{u_b\}.
\end{eqnarray*}

Having observed these expressions, we define Dyck partitions
$$\gamma_{\lambda,i} =\left\{ \begin{array}{ll}
\I_{(b-1)a, i} + \I_{a-2,1},  &\text{ if }1\leq i\leq b-1;\\ 
\I_{b(a-1), i+1}\setminus\{u_{b}\},  &\text{ if }b-1\leq i.
\end{array}\right.$$
(In particular, the two expressions coincide when $i=b-1$.) In the second case, one may check that the removal of cell $u_b$ subtracts
1 from position $j$ in $\I_{b(a-1),i+1}$, where 
\[ j=\left\{\begin{array}{ll}
i+2-r_{-b(a-1),i+1}-b+\ceil{b(a-1)/(i+1)}
&\text{ if }i+1-r_{-b(a-1),i+1}\geq b;\\
2i+2-r_{-b(a-1),i+1}-b+\ceil{b(a-1)/(i+1)}
&\text{ if }i+1-r_{-b(a-1),i+1}<b.\end{array}\right. \]
Next, define 
$$R(\gamma_{\lambda,i}) =\left\{ \begin{array}{ll}
\{\nu^m(\gamma_{\lambda,i}):0\leq m\leq i\},  
&\text{ if }1\leq i\leq b-2; \\
\{\nu^m(\gamma_{\lambda,i}):0\leq m\leq i+1\},  
&\text{ if }b-1\leq i. \end{array}\right.  $$
Finally, define
\begin{equation}\label{def:almost hook}
 \C_{(ab-b-1,b-1)}=\bigcup_{i=1}^{\infty} R(\gamma_{\lambda,i}).  
 \end{equation}
The notation $\C_{(ab-b-1,b-1)}$ is chosen so that it is 
compatible with Conjecture~\ref{strong_conj} below; note that
$a$, $b$, and $\lambda$ are uniquely determined by
$ab-b-1$ and $b-1$.


\begin{lemma}\label{ahk} 
Fix integers $a,b\geq 2$.
\begin{enumerate}
\item $\C_{(ab-b-1,b-1)}\subseteq \DP_{\ast,ab-2}
=\{\gamma\in\Par:\defc(\gamma)=ab-2\}$.  
\item For every integer $d\geq 2$, there exists a unique 
$\gamma\in\C_{(ab-b-1,b-1)}$ with $\dinv(\gamma)=d$.  
\end{enumerate}
\end{lemma}
\begin{proof} 
This is similar to the proof of Lemma~\ref{Delta_II},  
using Lemma~\ref{lem:propertiesnu} and the following Dyck vectors for 
$\gamma_{\lambda,i}$.  For $i$ in the range $1\leq i\leq b-1$, let 
$k=(b-1)a$ and $n=i+\ceil{k/i}+2$; then
\[ \DV_n(\gamma_{\lambda,i})=
(0,1,\underline{2}^{r_{-k,i}},\underline{1}^{i-r_{-k,i}},
 2\arrowupto(\ceil{k/i}-a+2),
 (\ceil{k/i}-a+2)\arrowupto\ceil{k/i}). \]
For $i\geq b$, let $k=b(a-1)$ and $n=i+\ceil{k/(i+1)}+3$; then
\[ \DV_n(\gamma_{\lambda,i})=
(0,1,\underline{2}^{r_{-k,i+1}},\underline{1}^{i+1-r_{-k,i+1}},
 2\arrowupto(\ceil{k/(i+1)}+1))+\epsilon, \]
where $\epsilon$ is the sequence with a 1 in position
\[ \left\{\begin{array}{ll}
r_{-k,i+1}+b+2 & \text{ if }i+1-r_{-k,i+1}\geq b, \\
r_{-k,i+1}+b-i+2& \text{ if }i+1-r_{-k,i+1}<b,
\end{array}\right.  \]
and zeroes elsewhere.  
\end{proof}

Part (2) of the previous lemma justifies the following definition.
\begin{definition}
For $a,b,i\geq 2$, we let $C_{(ab-b-1,b-1),i}$ 
denote the unique $\gamma\in\C_{(ab-b-1,b-1)}$ with $\dinv(\gamma)=i$.
\end{definition}

\begin{theorem}\label{almost_hook_chain}
For all $a,b\geq 2$ and all $n\geq 0$,
$\C_{(ab-b-1,b-1)}\cap\DP_n$ and $\C_{(ab-a-1,a-1)}\cap\DP_n$
are $n$-opposite to each other.
\end{theorem}

We prove this theorem in the following subsections.
Throughout, we fix $a,b\geq 2$ and define $k=ab-2$,
$S=(ab-b-1,b-1)$, and $S'=(ab-a-1,a-1)$.

\subsection{Exponent Pairs Appearing in $\C_{S}$ and $\C_{S'}$.}
\label{subsec:exp-pairs-almost-hook}


Define the sequence $A_{S}=(a_2,a_3,\ldots)$ by setting $a_i=\Delta(C_{S,i})$
for $i\geq 2$. 
Define the sequence $A_{S'}=(b_2,b_3,\ldots)$ by setting $b_i=\Delta(C_{S',i})$
for $i\geq 2$. 

For any $n\geq 0$, define
$$\begin{array}{lcl}
T_{n,k}&=&\{(\dinv(\gamma), \area_n(\gamma)):\gamma\in\C_{S}\cap\DP_n\};\\
T'_{n,k}&=&\{(\dinv(\gamma), \area_n(\gamma)):\gamma\in\C_{S'}\cap\DP_n\}. 
\end{array}$$
By Lemma~\ref{ahk},
$$\aligned
&T_{n,k}=\left\{\left(i,\binom{n}{2}-k-i\right):i\ge 2, a_i\le n\right\};\\
&T'_{n,k}=\left\{\left(i,\binom{n}{2}-k-i\right):i\ge 2, b_i\le n\right\}.
\endaligned$$
Theorem~\ref{almost_hook_chain} is equivalent to the assertion that
$(x,y)\in T_{n,k}$ iff $(y,x)\in T'_{n,k}$.

\begin{example}
Let $a=3$ and $b=4$, so $\lambda=(4,2,1)$, $\lambda'=(3,2,1,1)$, $k=10$,
$S=(7,3)$, and $S'=(8,2)$.
We compute
$$\aligned
A_S&=(\underbrace{11, 12}_{\Delta(R(\gamma_{\lambda,1}))},  \underbrace{\underline{8}^2, 9}_{\Delta(R(\gamma_{\lambda,2}))}, \underbrace{7, \underline{8}^4}_{\Delta(R(\gamma_{\lambda,3}))}, \underbrace{\underline{8}^3, \underline{9}^3}_{\Delta(R(\gamma_{\lambda,4}))}, \underbrace{\underline{9}^5, \underline{10}^2}_{\Delta(R(\gamma_{\lambda,5}))}, \underbrace{\underline{10}^7, 11}_{\Delta(R(\gamma_{\lambda,6}))}, 10, 11,\ldots)\\
&=(11,12,\underline{8}^2, 9, 7, \underline{8}^7, \underline{9}^8, \underline{10}^9, 11, 10, \underline{11}^{10},\ldots); \\
A_{S'}&=(\underbrace{10, 11}_{\Delta(R(\gamma_{\lambda',1}))}, \underbrace{7, \underline{8}^3}_{\Delta(R(\gamma_{\lambda',2}))}, \underbrace{\underline{8}^4, 9}_{\Delta(R(\gamma_{\lambda',3}))}, \underbrace{\underline{8}^2, \underline{9}^4}_{\Delta(R(\gamma_{\lambda',4}))}, \underbrace{\underline{9}^4, \underline{10}^3}_{\Delta(R(\gamma_{\lambda',5}))}, \underbrace{\underline{10}^6, \underline{11}^2}_{\Delta(R(\gamma_{\lambda',6}))}, \underbrace{\underline{11}^8, \underline{12}}_{\Delta(R(\gamma_{\lambda',7}))},11,12, \ldots)\\
&=(10,11,7,\underline{8}^7,9, \underline{8}^2, \underline{9}^8, \underline{10}^9, \underline{11}^{10},12,11,\underline{12}^{11},\ldots).
\endaligned$$
So we find, in accordance with Theorem~\ref{almost_hook_chain}, that
$$\begin{array}{lcl}T_{7,10}&=&\{(7,21-10-7)\}=\{(7,4)\};\\
        T'_{7,10}&=&\{(4, 21-10-4)\}=\{(4,7)\};\\
        T_{8,10} &=& \{(4,14),(5,13),...,(14,4)\}\setminus\{(6,12)\};\\
        T'_{8,10} &=& \{(4,14),(5,13),...,(14,4)\}\setminus\{(12,6)\};\\
        T_{9,10}&=&\{(4,22),(5,21),...,(22,4)\}=T_{9,10}';\\
        T_{10,10} &=& \{(4,31),(5,30),...,(33,2)\}\setminus\{(32,3)\};\\
        T'_{10,10} &=& \{(2,33),(3,32),...,(31,4)\}\setminus\{(3,32)\}.
\end{array}$$ 
\end{example}

\subsection{Proof of Theorem~\ref{almost_hook_chain}}
\label{subsec:prove-almost-hook}

We may assume $T_{n,k}$ and $T'_{n,k}$ are nonempty.
Define $s(\ell)=\ell+\ceil{(b-1)a/\ell}+1$
and $\tilde{s}(\ell)=\ell+\ceil{b(a-1)/\ell}+1$. Then let
\begin{flalign*}
&& x_1=\min\{i: a_i\le n\},\quad \quad & x_2=\max\{i: a_i\le n\}=\max\{i: a_i= n\},&&\\
&& y_1=\min\{i: b_i\le n\},\quad \quad & y_2=\max\{i: b_i\le n\}=\max\{i: b_i= n\},&&\\
&& \ell_1=\min\{\ell : s(\ell)\le n\},\quad &\ell_2=\max\{\ell : s(\ell)\le n\}=\max\{\ell : s(\ell)= n\}, && \\
&& \ell_1'=\min\{\ell:s(\ell)<n\}, \quad &\ell_2'=\min\{\ell>\ell_1': s(\ell)=n\} \quad\text{(if these exist)}.&&
\end{flalign*}
Define $\widetilde{\ell_1},\widetilde{\ell_2},\widetilde{\ell_1'},
\widetilde{\ell_2'}$ similarly by replacing $s(\ell)$ with 
$\tilde{s}(\ell)$ throughout. Special cases occur when
$\ell_1'$, $\ell_2'$, $\widetilde{\ell_1'}$, or $\widetilde{\ell_2'}$
are undefined, or when $\ell_1=\ell_1'$ or $\widetilde{\ell_1}
=\widetilde{\ell_1'}$. These cases, which are easier than the generic case 
discussed below, can be handled by arguments analogous to that leading
to~\eqref{degenerate case}.

Let $j$, $i'$ be two integers determined by $R(\gamma_{\lambda,i})=\{C_{S,j}, C_{S,j+1},\ldots,C_{S,j+i'}\}$.
Then $(a_j, a_{j+1},\ldots,a_{j+i'})$ has the form 
$(d,\ldots,d,d+1,\ldots,d+1)$. Hence $C_{S,x_1}=
\gamma_{\lambda,i}$ for some $i$. 
Using 
$$\Delta(\gamma_{\lambda,i})=
\begin{cases}
\Delta(\I_{(b-1)a, i})=i+\lceil \frac{ab-a}{i}\rceil +1,& \textrm{ if } i\le b-1, \\
\Delta(\I_{(a-1)a, i})=i+\lceil \frac{ab-b}{i+1}\rceil +2,& \textrm{ if } i\ge b, \\
\end{cases}$$
one may check that:

(i) If $a<b$, then
$$
\aligned
&\Delta(\gamma_{\lambda,1})\geq\cdots\geq\Delta(\gamma_{\lambda,a-1})>\Delta(\gamma_{\lambda,a})=a+b=\Delta(\gamma_{\lambda,b-1})<\Delta(\gamma_{\lambda,b})\leq \Delta(\gamma_{\lambda,b+1})\leq\cdots,\\
&\Delta(\gamma_{\lambda',1})\geq\cdots\geq\Delta(\gamma_{\lambda',a-2})>\Delta(\gamma_{\lambda',a-1})=a+b=\Delta(\gamma_{\lambda',b-2})<\Delta(\gamma_{\lambda',b-1})
\leq\Delta(\gamma_{\lambda',b})\leq\cdots,\\
\endaligned
$$
and $\Delta(\gamma_{\lambda,i})\le a+b$ when $a<i<b-1$,  $\Delta(\gamma_{\lambda',i})\le a+b$ when $a-1<i<b-2$. 

(ii) If $a=b$, then $\lambda=\lambda'$ and $\Delta(\gamma_{\lambda,a-2})>\Delta(\gamma_{\lambda,a-1})=2a<\Delta(\gamma_{\lambda,a})$.
$$\Delta(\gamma_{\lambda,1})\geq\cdots\geq\Delta(\gamma_{\lambda,a-2})>\Delta(\gamma_{\lambda,a-1})=2a<\Delta(\gamma_{\lambda,a})\leq \Delta(\gamma_{\lambda,a+1})\leq\cdots.$$

First consider the case $n\leq a+b$. If $a=b$ then 
$$\C_{(a^2-a-1,a-1)}\cap\DP_n=
\left\{\begin{array}{ll} \emptyset, & \text{ for }n<2a;\\
\{\gamma_{\lambda,a-1}\}, & \text{ for }n=2a, \end{array}\right.$$
and $\area_n(\gamma_{\lambda,a-1})=\dinv(\gamma_{\lambda,a-1})=\binom{a}{2}+1$ for $n=2a$,  hence $C_{(a^2-a-1,a-1)}\cap\DP_n$ is $n$-opposite to itself.

When $a<b$ it is enough to show that 
$\DP_n\cap\bigcup_{i=a}^{b-1} R(\gamma_{\lambda,i})$ and 
$\DP_n\cap\bigcup_{i=a-1}^{b-2} R(\gamma_{\lambda',i})$ 
are $n$-opposite to each other.
Since
\begin{eqnarray*}
\Delta(\nu^m(\gamma_{\lambda,i}))
&=& \Delta(\nu^m(\I_{(b-1)a, i} + \mathbb{I}_{a-2,1}))
   =\Delta(\nu^m(\I_{(b-1)a, i})) \\ & &
 \text{ for }a\leq i\leq b-1\mbox{ and }0\leq m\leq i; \\
\Delta(\nu^m(\gamma_{\lambda',i})) 
&=&\Delta(\nu^m(\I_{a(b-1), i+1}\setminus\{u_{a}\}))
  =\Delta(\nu^m(\I_{a(b-1), i+1}))\\ & & 
 \text{ for } a-1\leq i\leq b-2\mbox{ and }0\leq m\leq i+1, 
\end{eqnarray*}
the subsequence $(a_{x_1},a_{x_1+1},\ldots,a_{x_2})$ of $A_{S}$ 
and the subsequence $(b_{y_1},b_{y_1+1},\ldots,b_{y_2})$ of $A_{S'}$   
are of the same form~\eqref{general case} with $k=ab-a$. Since \eqref{general case} is equal to its reversal,  
it suffices to check that $C_{S,x_1}$ and $C_{S',y_2}$ form a pair of 
$n$-opposites, because the rest follows from the same argument as in the 
proof of Lemma~\ref{easylemma}.   

We have
$$\begin{array}{lcl}C_{S,x_1}=\mathbb{I}_{(b-1)a, a} + \mathbb{I}_{a-2,1};\\
C_{S',y_2}=\mathbb{I}_{a(b-1),b-1}\setminus\{u_{a}\}.\end{array}$$
 Hence
$$
\operatorname{dinv}(C_{S,x_1})=\operatorname{dinv}(\mathbb{I}_{(b-1)a, a})+1=\operatorname{area}_n(\mathbb{I}_{a(b-1), b-1})+1=
\operatorname{area}_n(C_{S',y_2})
$$
Since both $C_{S,x_1}$ and $C_{S',y_2}$ have the same deficit $ab-2$, the above equality implies  
$\operatorname{area}(C_{S,x_1})=\operatorname{dinv}_n(C_{S',y_2})$, thus we conclude that 
 $C_{S,x_1}$ and $C_{S',y_2}$ form a pair of $n$-opposites. 

Next consider the case $n>a+b$. 
Let $k=(b-1)a$ and $k'=b(a-1)$. 
Then the subsequence $(a_{x_1},a_{x_1+1},\ldots,a_{x_2})$ of $A_{S}$ 
has the form
\begin{equation}\label{almosthook case a}
    \underbrace{\rm o\cdots o}_{1+r_{-k,\ell_1}}
    \underbrace{*\cdots *}_{r^+_{k,\ell_1}}
    \cdots\cdots
    \underbrace{\rm o\cdots o}_{1+r_{-k,\ell_1'-1}}
    \underbrace{*\cdots*}_{r^+_{k,\ell_1'-1}}
    \underbrace{\rm
      o\cdots\cdots o
      }_{{\rm central\; block\; } B}
       \underbrace{*\cdots*}_{r^+_{k',\widetilde{\ell_2'}}}
        \underbrace{\rm o\cdots o}_{1+r_{-k',\widetilde{\ell_2'}+1}}
        \cdots\cdots
       \underbrace{*\cdots*}_{r^+_{k',\widetilde{\ell_2}-1}}
      \underbrace{\rm o\cdots o}_{1+r_{-k',\widetilde{\ell_2}}},
\end{equation}
and the subsequence $(b_{y_1},b_{y_1+1},\ldots,b_{y_2})$ of $A_{S'}$ has the form
\begin{equation}\label{almosthook case b}
    \underbrace{\rm o\cdots o}_{1+r_{-k',\widetilde{\ell_1}}}
    \underbrace{*\cdots *}_{r^+_{k',\widetilde{\ell_1}}}
    \cdots\cdots
    \underbrace{\rm o\cdots o}_{1+r_{-k',\widetilde{\ell_1'}-1}}
    \underbrace{*\cdots*}_{r^+_{k',\widetilde{\ell_1'}-1}}
    \underbrace{\rm
      o\cdots\cdots o
      }_{{\rm central\; block\; } B}
       \underbrace{*\cdots*}_{r^+_{k,\ell_2'}}
        \underbrace{\rm o\cdots o}_{1+r_{-k,\ell_2'+1}}
        \cdots\cdots
       \underbrace{*\cdots*}_{r^+_{k,\ell_2-1}}
      \underbrace{\rm o\cdots o}_{1+r_{-k,\ell_2}}.
\end{equation}
Using the facts proved in \S2.6 and \S2.7, the above two forms \eqref{almosthook case a} and \eqref{almosthook case a} are mutual reversal. Thus it suffices to check that $C_{S,x_1}$ and $C_{S',y_2}$ (resp.
$C_{S,x_2}$ and $C_{S',y_1}$) form a pair of $n$-opposites, 
which is straightforward as above. Finally, we can define an explicit 
bijection from $\C_{(ab-b-1,b-1)}\cap\DP_n$ to $\C_{(ab-a-1,a-1)}\cap\DP_n$
interchanging $\area_n$ and $\dinv$ by the same method used in
\S\ref{subsec:bij-Ck}.

\section{Conjectures}
\label{sec:stronger_conj}

In this section we propose conjectures that
generalize Theorems~\ref{hook_chain} and~\ref{almost_hook_chain}.

\subsection{Conjectured Decomposition of Level $k$ Objects into Chains.}
\label{subsec:conj-chain}

\begin{definition}
For any partition $\gamma=(\gamma_1,...,\gamma_v)$, 
let $H_\gamma=\{j \ :  \ \gamma_j=\gamma_{j+1}=\gamma_{j+2}>0\}$,  
and define $$\rho(\gamma)=\left\{\begin{array}{ll} \gamma_i, 
&\text{ if   $\nu(\alpha)$ is defined and }H_\gamma\neq\emptyset\text{ and }i=\max(H_\gamma);\\
\gamma_1, &\text{ if  $\nu(\alpha)$ is defined and }H_\gamma=\emptyset; \\  
0, &\text{ if $\nu(\alpha)$ is not defined.} \\
\end{array}   \right.$$
Let $R(\gamma)=\{\nu^m(\gamma):0\leq m\leq\rho(\gamma) \}
\subseteq\Par$. One may check that $\nu^m(\gamma)$ is well-defined 
for $0\leq m\leq \rho(\gamma)$.
\end{definition}

\begin{conjecture}\label{wdssPar}
For any $\lambda\in\Par(k)$, there exists a set 
$\{\gamma_{\lambda,i}:i\in\Z_{>0}\}$ of Dyck partitions
with $|\gamma_{\lambda,1}|<|\gamma_{\lambda,2}|<\cdots$ 
satisfying the following conditions:
\begin{enumerate}
\item $\dpmap_k(\dyckmap(\lambda))=\gamma_{\lambda,i}$ for some $i$. 

\item There exists $t\in\Z_{>0}$ such that
$i\mapsto \Delta(\gamma_{\lambda,i})$ is weakly decreasing 
on $\{1,2,\ldots,t\}$ and weakly increasing on $\{t,t+1,\ldots\}$.

\item $\defc(\gamma_{\lambda,i})$ is a constant not depending on $i$
(namely, $\sum_{r=1}^{\ell(\lambda)-1} (-1+\lambda_r)
\sum_{s=r+1}^{\ell(\lambda)} \lambda_s$, which is the number of pairs
of cells in $\dg(\lambda)$ that are not in the same row or column).

\item For any integer $d\geq |\gamma_{\lambda,1}|$, there exists a 
unique Dyck partition $\gamma\in\bigcup_{i\geq 1} R(\gamma_{\lambda,i})$ 
such that $|\gamma|=d$.

\item For all $n\geq 0$, $\DP_n\cap\bigcup_{i\geq 1} R(\gamma_{\lambda,i})$ 
and $\DP_n\cap\bigcup_{i\geq 1} R(\gamma_{\lambda',i})$ are $n$-opposite 
to each other.
\end{enumerate}
\end{conjecture}

Theorems~\ref{hook_chain} and~\ref{almost_hook_chain} imply that 
Conjecture~\ref{wdssPar} holds for partitions of hook shape or 
almost hook shape. We also checked that Conjecture~\ref{wdssPar} 
holds for $\lambda=(3,3)$, $(4,3)$, $(3,3,1)$, $(5,3)$, $(4,3,1)$, 
$(3,3,1,1)$, and the conjugates of these partitions.  

\subsection{Strong Conjecture on the Structure of Level $k$ Objects}
\label{subsec:strong-conj}

We now propose another conjecture, which is stronger than  
Conjectures~\ref{conj:DP} and~\ref{wdssPar}. 
When we construct $\DP_{\mu}$ and $\overline{\DP}_{\mu}$ for 
a partition $\mu$, we want to include the most natural Dyck vector 
from which $\mu$ can be reconstructed. This Dyck vector is defined as follows.

\begin{definition}\label{def:natural}
For any partition $\mu$, let $\zeta=\zeta(\mu)=\mu_1+\ell(\mu)$.  
Write $\mu=(\underline{k}^{e_k},\ldots,
\underline{2}^{e_2},\underline{1}^{e_1})$ with all $e_j\geq 0$ and $e_k>0$,
and define 
\begin{equation}\label{eq:nat-DV}
\gamma_{\mu}=\dpmap_{\zeta+1}(0,0,\underline{1}^{e_1},0,\underline{1}^{e_2},
\ldots,0,\underline{1}^{e_k}).
\end{equation}
Equivalently, we can write 
$$\gamma_{\mu}=\I_{0,\zeta}\setminus \{u_{\mu_1+\ell(\mu)},
 u_{\mu_2+\ell(\mu)-1},u_{\mu_3+\ell(\mu)-2},\ldots\}.$$
For $\mu=(0)$, we set $\zeta(\mu)=0$ and $\gamma_{\mu}=(0)$.
\end{definition}


\begin{example}
Let $\mu=(2,2,1,1,1)\in\Par(7)$. Here $\zeta(\mu)=2+5=7$
and $\dvmap_8(\gamma_{\mu})=(0,0,1,1,1,0,1,1)$, so
$\gamma_{\mu}=(6,5,5,3,2,1,1,0)$. The following diagram shows
how we can compute $\gamma_{\mu}$ by removing cells 
$u_7,u_6,u_4,u_3,u_2$ from the diagram of $\I_{0,7}=\Delta_8$.
\[ \tableau{
{}&{}&{}&{}&{}&{}&{u_7} \\
{}&{}&{}&{}&{}&{u_6} \\
{}&{}&{}&{}&{} \\
{}&{}&{}&{u_4} \\
{}&{}&{u_3} \\
{}&{u_2} \\
{} } \]
\end{example}

\begin{lemma}\label{kappa_natural}
For all $\mu\in\Par$, $\defc(\gamma_{\mu})=|\mu|$.  
\end{lemma}
\begin{proof}
Let $v=\dvmap_{\zeta+1}(\gamma_{\mu})$ be the Dyck vector
shown in~\eqref{eq:nat-DV}. This vector has $e_1+e_2+\cdots+e_k=\ell(\mu)$
ones and $\zeta+1-\ell(\mu)=\mu_1+1$ zeroes. We directly calculate
\[ \dinv(v)=\binom{\ell(\mu)}{2}+\binom{\mu_1+1}{2}
           +\sum_{i=1}^{\mu_1} e_i(\mu_1-i). \]
The last sum is $\mu_1\ell(\mu)-|\mu|$.
Also, $\area(v)=e_1+e_2+\cdots+e_k=\ell(\mu)$, so
\[ \defc(\gamma_{\mu})=\defc(v)=
\binom{\ell(\mu)+\mu_1+1}{2}-\binom{\ell(\mu)}{2}-\binom{\mu_1+1}{2}
 -\mu_1\ell(\mu)-\ell(\mu)+|\mu|=|\mu|. \]
\end{proof} 


We can now state our strengthened version of Conjecture~\ref{conj:DP}.

\begin{conjecture}\label{strong_conj}
For all integers $k\geq 0$ and all $\mu\in\Par(k)$,
there exist two (possibly identical) collections $\C_{\mu}$
and $\overline{\C}_{\mu}$
of Dyck partitions satisfying the following conditions:
\begin{enumerate}
\item[(1)] The sets $\C_{\mu}$ are pairwise disjoint,
 and the sets $\overline{\C}_{\mu}$ are pairwise disjoint.
\item[(2)] $\DP_{\ast,k}=\bigcup_{\mu\in\Par(k)} \C_{\mu}
   =\bigcup_{\mu\in\Par(k)} \overline{\C}_{\mu}$.
\item[(3)] For all $m\geq 0$, $\nu^m(\gamma_{\mu})\in\C_{\mu}$.
\item[(4)] For each $\mu\in\Par(k)$, there exists a unique
$\mu^*\in\Par(k)$ such that $\C_{\mu}=\overline{\C}_{\mu^*}$
and $\C_{\mu^*}=\overline{\C}_{\mu}$ (note $\mu^*$ may not equal 
the conjugate partition $\mu'$).
\item[(5)] For all $\mu\in\Par(k)$ and all $i\geq\ell(\mu^*)$,
there exists a unique partition $C_{\mu,i}\in\C_{\mu}$ with 
$\dinv(C_{\mu,i})=i$.  Moreover, there exists an increasing integer sequence $i_1 < i_2 < \ldots$ such that   
$\C_{\mu}=\bigcup_{j\geq 1} R(C_{\mu,i_j})$ and 
the map $j\mapsto\Delta(C_{\mu,i_j})$ is weakly decreasing
on $\{1,2,\ldots,t\}$ 
and weakly increasing on $\{t,t+1,\ldots\}$ for some $t\geq1$.
\item[(6)] For sufficiently large $n$ (more precisely, for
$n\geq\max(\zeta(\mu)+2,\zeta(\mu^*)+2)$), we have
\[ \ell(\mu^*)\leq i\leq \binom{n}{2}-k-\ell(\mu) \]
if and only if there exists a (unique) $\gamma\in\C_{\mu}$
with $\Delta(\gamma)\leq n$ and $\dinv(\gamma)=i$.
\item[(7)] If $\gamma\in\C_{\mu}$ and $\nu(\gamma)$ is defined,
 then $\nu(\gamma)\in\C_{\mu}$.
\item[(8)] If $\dpmap_n(\dyckmap(\lambda))\in\C_{\mu}$
for some $\lambda\in\Par(n)$, then $\dpmap_n(\dyckmap(\lambda'))
\in\overline{\C}_{\mu}$.
\item[(9)] For all $n\geq 0$, $\C_{\mu}\cap\DP_n$ and 
$\overline{\C}_{\mu}\cap\DP_n$ are $n$-opposite to each other.  
\end{enumerate}
\end{conjecture}

Evidently, Conjecture~\ref{strong_conj} implies Conjecture~\ref{conj:DP}
(take $\DP_{\mu}=\C_{\mu}$ and $\overline{\DP}_{\mu}=\overline{\C}_{\mu}$).
Our next goal is to show that condition (2) in 
Conjecture~\ref{strong_conj} is already implied by the other conditions.
Write $p_{\leq d}(k)$ for the number of partitions $\lambda$ with
$|\lambda|=k$ and $\ell(\lambda)\leq d$.  Define 
$$\DP_{\ast,k}(d) =\{\gamma\in \DP_{\ast,k}:\dinv(\gamma)=d\}
 =\{\gamma\in\Par:\defc(\gamma)=k\mbox{ and }\dinv(\gamma)=d\}.$$
For any subset $\C\subseteq\DP_{\ast,k}$, define $\C(d)=\C\cap \DP_{\ast,k}(d)$.

\begin{lemma}
For all $d,k\geq 0$, $|\DP_{\ast,k}(d)|=p_{\leq d}(k)$. 
\end{lemma}
\begin{proof}
In~\cite[Theorem 3]{LW}, Loehr and Warrington showed that the two
statistics $\dinv$ and $\ell$ (number of parts) have the same
distribution on $\Par(n)$ for all $n$; more specifically,
\[ \sum_{\lambda\in\Par} q^{|\lambda|}t^{\ell(\lambda)}
 =\prod_{i=1}^{\infty} \frac{1}{1-tq^i}
 =\sum_{\lambda\in\Par} q^{|\lambda|}t^{\dinv(\lambda)}. \]
Replacing $t$ by $t/q$ in this identity, we get
\[ \sum_{\lambda\in\Par} q^{|\lambda|-\ell(\lambda)}t^{\ell(\lambda)}
  =\sum_{\lambda\in\Par} q^{\defc(\lambda)}t^{\dinv(\lambda)}. \]
The coefficient of $q^kt^d$ on the right side is $|\DP_{\ast,k}(d)|$.
The coefficient of $q^kt^d$ on the left side is the number of 
partitions of $k+d$ into exactly $d$ nonzero parts. By decreasing
each of these $d$ parts by $1$, we see that this is the number
of partitions of $k$ into at most $d$ nonzero parts, namely $p_{\leq d}(k)$.
\end{proof}

\begin{lemma}
In Conjecture~\ref{strong_conj}, part {\rm(2)} follows from 
parts {\rm(1)}, {\rm(4)}, and {\rm(5)}.
\end{lemma}
\begin{proof}
Parts (1), (4), and (5) of Conjecture~\ref{strong_conj}
imply  that $\bigcup_{\mu\in\Par(k)} \C_{\mu}
=\bigcup_{\mu\in\Par(k)} \overline{\C}_{\mu}$
and that $\C=\bigcup_{\mu\in\Par(k)} \C_{\mu}$
satisfies $|\C(d)|=p_{\leq d}(k)=|\DP_{\ast,k}(d)|$.  
Thus, $\DP_{\ast,k}$ must be the union of the sets $\C_{\mu}$ (and the union 
of the sets $\overline{\C}_{\mu}$) as $\mu$ varies through $\Par(k)$.
\end{proof}

Conjecture~\ref{strong_conj}(5) now implies that each collection
$\C_{\mu}$ has the form $\{C_{\mu,i}:i\geq \ell(\mu^*)\}$ 
with $\dinv(C_{\mu,i})=i$.

\begin{lemma}\label{almost_all}
For all $\mu\in\Par$ and $m\in\Z_{\geq 0}$,
\begin{enumerate}
\item  $\nu^m(\gamma_{\mu})$ is defined;
\item $\dinv(\nu^m(\gamma_{\mu}))=\binom{\mu_1+\ell(\mu)+1}{2}
 -|\mu|-\ell(\mu)+m$; 
\item $(\Delta(\nu^1(\gamma_{\mu})), \Delta(\nu^2(\gamma_{\mu})), \ldots)
 =(\underline{(\zeta+2)}^{\zeta+1}, \underline{(\zeta+3)}^{\zeta+2},\ldots)$, where $\zeta=\zeta(\mu)$.
\end{enumerate}
\end{lemma}
\begin{proof}
One can explicitly describe $\nu^m(\gamma_{\mu})$, namely 
$$\{\nu^m(\gamma_{\mu}) : m\in\Z_{\geq 0}\}
=\bigcup_{i\geq \zeta(\mu)} 
R(\I_{0,i}\setminus\{u_{\mu_1+\ell(\mu)},u_{\mu_2+\ell(\mu)-1},\ldots\}),$$
from which all the statements readily follow.
\end{proof}

\begin{theorem}\label{mainthm2}
Conjecture~\ref{strong_conj} holds for all $k\leq 9$.
\end{theorem}
\begin{proof}
We first note that   $\nu^{m}(\gamma_\mu)$ are all distinct as we let $\mu$ and $m$ vary, because the operator $\nu$ is one-to-one and $\nu^{-1}(\gamma_\mu)$ are not defined. Since $|\DP_{\ast,k}(d)|=p_{\leq d}(k)$,   we have 
$$\DP_{\ast,k}(d)=\left\{\nu^m(\gamma_{\mu}) \ : \mu\in\Par(k)\mbox{ and }
 m=d-\binom{\mu_1+\ell(\mu)+1}{2}+k+\ell(\mu)\right\}$$ for each $d\geq \binom{k+2}{2}$.
 
In particular, any partition $\gamma$ with 
$\defc(\gamma)=k$ and $\dinv(\gamma)=d$ has the form
$\nu^m(\gamma_{\mu})$ for some $\mu$ and $m$.
Hence, for fixed $k$, there are only finitely many objects 
in $\DP_{\ast,k}$ that are not of the form $\nu^m(\gamma_{\mu})$. 
So if $\gamma$ is not of the form $\nu^m(\gamma_{\mu})$, then 
$\Delta(\gamma)\leq n_k$ for some constant $n_k$ depending only on $k$.  
One readily sees that $n_k$ can be any integer greater than $k+1$.

So if Conjecture~\ref{strong_conj}(6) holds for $n=n_k+1$, then 
it also holds for $n>n_k+1$ by Lemma~\ref{almost_all}. In turn, 
Conjecture~\ref{strong_conj}(9) holds for all $n\geq n_k+1$. 
Thus, to prove Conjecture~\ref{strong_conj} for a given $k$, 
it suffices to prove the conjecture with conditions (6) and (9) replaced by
the following conditions.
\begin{enumerate}
\item[(6$'$)] 
For $\max(\zeta(\mu)+1,\zeta(\mu^*)+1)\leq n\leq n_k+1$,
$$\ell(\mu^*)\leq i\leq \binom{n}{2}-k-\ell(\mu)$$
if and only if  there exists a (unique) $\gamma\in\C_{\mu}$ with
$\Delta(\gamma)\leq n$ and $\dinv(\gamma)=i$.
\item[(9$'$)]
For all $n\leq n_k$, $\C_{\mu}\cap\DP_n$ and $\overline{\C}_{\mu}\cap\DP_n$
are $n$-opposite to each other.
\end{enumerate}

Therefore, for fixed $k$, Conjecture~\ref{strong_conj} can be checked in 
finitely many steps by examining finitely many objects. By exhaustive 
computations, we have checked Conjecture~\ref{strong_conj} for $k\leq 9$. 
The list of $\C_{\mu}$ and $\overline{\C}_{\mu}$
for $k\leq 6$ is given in the appendix below,
along with a link to the list of $\C_{\mu}$ and $\overline{\C}_{\mu}$
for $k\leq 9$.  
\end{proof}

In fact, our computations show that for all $k\leq 4$,
the collections $\C_{\mu}$ and $\overline{\C}_{\mu}$
with $\mu\in\Par(k)$ are uniquely
determined by the conditions in Conjecture~\ref{strong_conj}.

\section{Appendix}\label{app}
Below is the list of all $\C_{\mu}$ and $\overline{\C}_{\mu}$
with $k=|\mu|\le 6$.  The list of $\C_{\mu}$ and $\overline{\C}_{\mu}$
for $k\leq 9$ may be found on the website
\verb-www.oakland.edu/~li2345/List_of_C_mu.pdf-.

{\tiny

$k=0$:
\begin{enumerate}
\item  $\mathcal{C}_{(0)}=\{\nu^{(m)}(\gamma_{(0)}) :  m\in[0,\infty) \}=\{\nu^{(m)}((0)) :  m\in[0,\infty) \}=\overline{\mathcal{C}}_{(0)}$
\end{enumerate}

$k=1$:
\begin{enumerate}
\item  $\mathcal{C}_{(1)}=\{\nu^{(m)}(\gamma_{(1)}) :  m\in[0,\infty) \}=\{\nu^{(m)}((1,1)) :  m\in[0,\infty) \}=\overline{\mathcal{C}}_{(1)}$
\end{enumerate}

$k=2$:
\begin{enumerate}
 \setlength\itemsep{.5em}
\item  $\mathcal{C}_{(2)}=R((1,1,1)) \cup \{\nu^{(m)}(\gamma_{(2)}) :  m\in[0,\infty) \}
 =\{(1,1,1),(4)\}\cup\{\nu^{(m)}((2,2,1)) :  m\in[0,\infty) \}
=\overline{\mathcal{C}}_{(2)}$ 

\item  $\mathcal{C}_{(1,1)}=\{\nu^{(m)}(\gamma_{(1,1)}) :  m\in[0,\infty) \}=\{\nu^{(m)}((2,1,1)) :  m\in[0,\infty) \}=\overline{\mathcal{C}}_{(1,1)}$

\end{enumerate}

$k=3$:
\begin{enumerate}
 \setlength\itemsep{.5em}
\item  $\mathcal{C}_{(3)}=R((1,1,1,1))\cup R((2,2,2)) \cup \{\nu^{(m)}(\gamma_{(3)}) :  m\in[0,\infty) \}$
 
 ${\;}=\{(1,1,1,1),(5)\}\cup\{(2,2,2),(4,1,1,1),(5,3)\}\cup\{\nu^{(m)}((3,3,2,1)) :  m\in[0,\infty) \}
 =\overline{\mathcal{C}}_{(3)}$ 

\item  $\mathcal{C}_{(2,1)}=R((3,1,1,1))\cup \{\nu^{(m)}(\gamma_{(2,1)}) :  m\in[0,\infty) \}
=\{(3,1,1,1),(5,2)\}\cup \{\nu^{(m)}((3,3,1,1)) :  m\in[0,\infty) \}
=\overline{\mathcal{C}}_{(1,1,1)}$ 

\item  $\mathcal{C}_{(1,1,1)}=R((2,1,1,1))\cup \{\nu^{(m)}(\gamma_{(1,1,1)}) :  m\in[0,\infty) \}$

${\;\;\;\;\;\;}=\{(2,1,1,1),(5,1)\}\cup \{\nu^{(m)}((3,2,1,1)) :  m\in[0,\infty) \}
=\overline{\mathcal{C}}_{(2,1)}$ 
\end{enumerate}

$k=4$:
\begin{enumerate}
 \setlength\itemsep{.5em}
\item  $\mathcal{C}_{(4)}=R((1,1,1,1,1))\cup R((2,2,2,1)) \cup R((3,3,3,1))\cup \{\nu^{(m)}(\gamma_{(4)}) :  m\in[0,\infty) \}$

$ {\;}=\{(1,1,1,1,1),(6)\} \cup\{(2,2,2,1),(5,1,1,1),(5,4)\}
\cup\{(3,3,3,1),(5,2,2,2),(5,4,1,1,1),(6,4,3)\}$

${\;\;\;\;\;}\cup\{\nu^{(m)}((4,4,3,2,1)) :  m\in[0,\infty) \}\
 =\overline{\mathcal{C}}_{(4)}$ 

\item  $\mathcal{C}_{(3,1)}=R((2,2,1,1))\cup R((3,3,3))\cup \{\nu^{(m)}(\gamma_{(3,1)}) :  m\in[0,\infty) \}$

${\;\;\;\,}=R((2,2,1,1))\cup R((3,3,3))\cup  \{\nu^{(m)}((4,4,3,1,1)) :  m\in[0,\infty) \}
=\overline{\mathcal{C}}_{(2,2)}$ 

\item  $\mathcal{C}_{(2,2)}=R((2,1,1,1,1))\cup \{\nu^{(m)}(\gamma_{(2,2)}) :  m\in[0,\infty) \}=R((2,1,1,1,1))\cup  \{\nu^{(m)}((3,2,2,1)) :  m\in[0,\infty) \}
=\overline{\mathcal{C}}_{(3,1)}$ 

\item  $\mathcal{C}_{(2,1,1)}=R((3,2,1,1,1))\cup R((4,3,1,1,1))\cup \{\nu^{(m)}(\gamma_{(2,1,1)}) :  m\in[0,\infty) \}$

${\;\;\;\;\;\;}=R((3,2,1,1,1))\cup R((4,3,1,1,1))\cup  \{\nu^{(m)}((4,4,2,1,1)) :  m\in[0,\infty) \}
=\overline{\mathcal{C}}_{(1,1,1,1)}$

\item  $\mathcal{C}_{(1,1,1,1)}=R((3,1,1,1,1))\cup R((4,2,1,1,1))\cup \{\nu^{(m)}(\gamma_{(1,1,1,1)}) :  m\in[0,\infty) \}$

${\;\;\;\;\;\;\;\;\,}=R((3,1,1,1,1))\cup R((4,2,1,1,1))\cup  \{\nu^{(m)}((4,3,2,1,1)) :  m\in[0,\infty) \}
=\overline{\mathcal{C}}_{(2,1,1)}$
\end{enumerate}

$k=5$:
\begin{enumerate}
 \setlength\itemsep{.5em}
\item  $\mathcal{C}_{(5)}=\overline{\mathcal{C}}_{(5)}$ is given in \eqref{def:hook}.

\item  $\mathcal{C}_{(4,1)}=R((4,1,1,1,1))\cup R((3,3,2,2))\cup R((4,4,4,2))\cup \{\nu^{(m)}(\gamma_{(4,1)}) :  m\in[0,\infty) \}
=\overline{\mathcal{C}}_{(2,2,1)}$ 

\item  $\mathcal{C}_{(3,2)}=R((2,2,1,1,1))\cup R((3,3,1,1,1))\cup R((4,4,1,1,1)) \cup \{\nu^{(m)}(\gamma_{(3,2)}) :  m\in[0,\infty) \}  
= \overline{\mathcal{C}}_{(3,2)}$ 

\item  $\mathcal{C}_{(2,2,1)}=R((2,1,1,1,1,1))\cup R((3,2,2,2))\cup \{\nu^{(m)}(\gamma_{(2,2,1)}) :  m\in[0,\infty) \}
=\overline{\mathcal{C}}_{(4,1)}$ 

\item  $\mathcal{C}_{(3,1,1)}=R((3,1,1,1,1,1))\cup R((4,2,2,1,1))\cup  R((5,4,2,2,2))\cup \{\nu^{(m)}(\gamma_{(3,1,1)}) :  m\in[0,\infty) \}
=\overline{\mathcal{C}}_{(3,1,1)}$

\item  $\mathcal{C}_{(2,1,1,1)}=R((3,2,1,1,1,1))\cup R((4,3,1,1,1,1))\cup R((5,3,2,1,1,1))\cup R((5, 4, 3, 1, 1, 1))\cup \{\nu^{(m)}(\gamma_{(2,1,1,1)}) :  m\in[0,\infty) \}
=\overline{\mathcal{C}}_{(2,1,1,1)}$

\item  $\mathcal{C}_{(1,1,1,1,1)}=R((4,2,1,1,1,1))\cup R((4,3,2,1,1,1))\cup R((5,4,2,1,1,1))\cup \{\nu^{(m)}(\gamma_{(1,1,1,1,1)}) :  m\in[0,\infty) \}
=\overline{\mathcal{C}}_{(1,1,1,1,1)}$
\end{enumerate}

$k=6$:
\begin{enumerate}
 \setlength\itemsep{.5em}
\item  $ \mathcal{C}_{(6)}=\overline{\mathcal{C}}_{(6)}$ is given in \eqref{def:hook}.

\item  $\mathcal{C}_{(5,1)} =\overline{\mathcal{C}}_{(3,3)}$ is given in \eqref{def:almost hook} for $(a,b)=(4,2)$ (almost hook shape).

\item  $\mathcal{C}_{(3,3)} =\overline{\mathcal{C}}_{(5,1)}$ is given in \eqref{def:almost hook} for $(a,b)=(2,4)$ (almost hook shape).

\item  $\mathcal{C}_{(4,2)}=R((3, 1, 1, 1, 1, 1, 1))\cup R((4, 2, 2, 2, 1))\cup R((4, 4, 4, 1, 1)) \cup  \{\nu^{(m)}(\gamma_{(4,2)}) :  m\in[0,\infty) \}  
= \overline{\mathcal{C}}_{(4,1,1)}$ 

\item  $\mathcal{C}_{(4,1,1)}=R((2, 2, 1, 1, 1, 1))\cup R((3, 3, 2, 1, 1))\cup R((4, 4, 3, 3)) \cup R((5, 5, 5, 3, 1))\cup \{\nu^{(m)}(\gamma_{(4,1,1)}) :  m\in[0,\infty) \}
=\overline{\mathcal{C}}_{(4,2)}$

\item  $\mathcal{C}_{(3,2,1)}=R((3, 2, 1, 1, 1, 1, 1))\cup R((5, 3, 1, 1, 1, 1))\cup  R((5, 3, 3, 1, 1, 1))\cup R((5, 5, 3, 1, 1, 1))$

${\;\;\;\;\;\;\;\;\;\;}\cup \{\nu^{(m)}(\gamma_{(3,2,1)}) :  m\in[0,\infty) \}
=\overline{\mathcal{C}}_{(3,1,1,1)}$

\item  $\mathcal{C}_{(3,1,1,1)}=R((4, 1, 1, 1, 1, 1))\cup R((4, 2, 2, 1, 1, 1))\cup  R((4, 4, 2, 1, 1, 1))
\cup R((5, 4, 2, 2, 1, 1)) \cup R((6, 5, 4, 2, 2, 2))$

${\;\;\;\;\;\;\;\;\;\;\;\;\;} \cup \{\nu^{(m)}(\gamma_{(3,1,1,1)}) :  m\in[0,\infty) \}
=\overline{\mathcal{C}}_{(3,2,1)}$

\item  $\mathcal{C}_{(2,2,2)}=R((3, 3, 1, 1, 1, 1))\cup  \{\nu^{(m)}(\gamma_{(2,2,2)}) :  m\in[0,\infty) \}
=\overline{\mathcal{C}}_{(2,2,1,1)}$

\item  $\mathcal{C}_{(2,2,1,1)}=R((3, 2, 2, 1, 1))\cup R((4, 3, 3, 3))\cup \{\nu^{(m)}(\gamma_{(2,2,1,1)}) :  m\in[0,\infty) \}
=\overline{\mathcal{C}}_{(2,2,2)}$

\item  $\mathcal{C}_{(2,1,1,1,1)}=R((4,2,1,1,1,1,1))\cup R((4,3,2,1,1,1,1))\cup R((5, 4, 2, 1, 1, 1,1))
\cup R((5,4,3,2,1,1,1))\cup R((6,5,3,2,1,1,1))$

${\;\;\;\;\;\;\;\;\;\;\;\;\;\;\;}\cup  R((6,5,4,3,1,1,1))\cup  \{\nu^{(m)}(\gamma_{(2,1,1,1,1)}) :  m\in[0,\infty) \}
=\overline{\mathcal{C}}_{(2,1,1,1,1)}$

\item  $\mathcal{C}_{(1,1,1,1,1,1)}=R((4,3,1,1,1,1,1))\cup R((5,3,2,1,1,1,1))\cup R((5, 4, 3, 1, 1, 1,1))
\cup R((6,4,3,2,1,1,1))$

${\;\;\;\;\;\;\;\;\;\;\;\;\;\;\;\;\;\;}\cup R((6,5,4,2,1,1,1))\cup  \{\nu^{(m)}(\gamma_{(1,1,1,1,1,1)}) :  m\in[0,\infty) \}=\overline{\mathcal{C}}_{(1,1,1,1,1,1)}$

\end{enumerate}
}

\end{document}